\documentclass{article}
\usepackage[utf8]{inputenc}
\usepackage[english]{babel}
\usepackage[T1]{fontenc}
\usepackage{lmodern}
\usepackage[a4paper]{geometry}
\usepackage{graphicx}
\usepackage{onlyamsmath}
\usepackage{amsfonts} 
\usepackage{amssymb}
\usepackage{ntheorem}
\usepackage{tikz}
\usepackage{verbatim}

\usepackage[ruled,vlined]{algorithm2e}

\newcommand{\be}{\begin{equation}}
  \newcommand{\ee}{\end{equation}}

\newcommand{\eps}{\varepsilon}

\newcommand{\E}{\mathbb{E}}
\newcommand{\bbS}{\mathbb{S}}

\newcommand{\Pro}{\mathbb{P}}
\newcommand{\R}{\mathbb{R}}

\def\fref#1{{\rm (\ref{#1})}}

\newcommand{\calA}{\mathcal{A}}
\newcommand{\calB}{\mathcal{B}}

\newcommand{\calE}{\mathcal{E}}
\newcommand{\calF}{\mathcal{F}}
\newcommand{\calG}{\mathcal{G}}
\newcommand{\calH}{\mathcal{H}}
\newcommand{\calI}{\mathcal{I}}

\newcommand{\calL}{\mathcal{L}}

\newcommand{\calQ}{\mathcal{Q}}

\newcommand{\calU}{\mathcal{U}}

\newtheorem{thm}{Theorem}[section]

\newtheorem{prop}{Proposition}[section]

\newtheorem{lemma}{Lemma}[section]

\theoremstyle{nonumberplain}
\theorembodyfont{\upshape}
\newtheorem{proof}{Proof}

\title{A Monte Carlo Method for 3D Radiative Transfer Equations with Multifractional Singular Kernels}

\author{Christophe Gomez\thanks{Aix Marseille Univ, CNRS, I2M, Marseille, France, christophe.gomez@univ-amu.fr} \and Olivier Pinaud\thanks{Department of Mathematics, Colorado State University, Fort Collins CO, pinaud@math.colostate.edu}}

\begin{document}

\maketitle

\begin{abstract}
  We propose in this work a Monte Carlo method for three dimensional scalar radiative transfer equations with non-integrable, space-dependent scattering kernels. 
  Such kernels typically account for long-range statistical features, and arise for instance in the context of wave propagation in turbulent atmosphere, geophysics, and medical imaging in the peaked-forward regime. In contrast to the classical case where the scattering cross section is integrable, which results in a non-zero mean free time, the latter here vanishes. This creates numerical difficulties as standard Monte Carlo methods based on a naive regularization exhibit large jump intensities and an increased computational cost. We propose a method inspired by the finance literature based on a small jumps - large jumps decomposition, allowing us to treat the small jumps efficiently and reduce the computational burden. We demonstrate the performance of the approach with numerical simulations and provide a complete error analysis.  The multifractional terminology refers to the fact that the high frequency contribution of the scattering operator is a fractional Laplace-Beltrami operator on the unit sphere with space-dependent index.

\end{abstract} 

\begin{flushleft}
\end{flushleft}

\section{Introduction}

Radiative transfer models have been used for more than a century to describe wave energy propagation through complex/random media \cite{jeans, chandrasekhar}, as well as neutron transport \cite{Lux, Spanier}, heat transfer \cite{viskanta}, and are still an active area of research in astrophysics, geophysics, and optical tomography \cite{louvin, margerin2, noebauer, powell} for instance. In this work, we propose a new Monte Carlo (MC) method to simulate the following radiative transfer equation (RTE) 
\begin{equation}\label{RTE}
\left\{ \begin{split}
\partial_t u + \hat k \cdot \nabla_x u  & = \calQ u ,\\
u(t=0,x,\hat k) & = u_0(x,\hat k),
\end{split}
\qquad (t,x,\hat k)\in (0,\infty)\times \R^3\times \mathbb{S}^2, \right.
\end{equation}
where $\mathbb{S}^2$ denotes the unit sphere in $\R^3$, and $u$ is the wave energy density in the context of wave propagation or a particle distribution function in the context of neutronics. The scattering operator $\calQ$ has the standard form
\begin{equation}\label{defQ}
  (\calQ u)(x,\hat k) = \lambda(x) \int_{\mathbb{S}^2} \Phi(x, |\hat p- \hat k|)  (u(x,\hat p)-u(x,\hat k))\sigma(d \hat p),
  \end{equation}
for $\sigma(d \hat p)$ the surface measure on $\mathbb{S}^2$, $\Phi$ the scattering kernel, and $\lambda>0$ a function modeling the support of the scattering process. Regions where $\lambda(x)=0$ are homogeneous and $u$ undergoes free transport. MC methods have long be used for the resolution of (\ref{RTE}), see e.g. \cite{Lapeyre,Spanier}. The originality and difficulty in our work lies in the fact that we consider situations where the mean free time $t_0$ associated with $\calQ$ vanishes in the scattering regions, that is
\begin{equation}\label{eq:mft}
\frac{1}{t_0(x)} = \lambda(x)\int_{\mathbb{S}^2} \Phi(x,|\hat k - \hat p|) \sigma(d \hat p) = +\infty, \qquad \textrm{where} \quad \lambda(x)>0,\end{equation}
and as a consequence the standard MC representations of $u$ do not apply. Such a scenario arises for instance in the context of highly peaked-forward light scattering in biological tissues and in turbulent atmosphere, or more generally in the context of wave propagation in random media with long-range correlations that we describe below. In this paper we write $\Phi$ as 
\begin{equation}\label{def:rho} 
\Phi(x,|\hat p- \hat k|) := \frac{a(|\hat p- \hat k|)}{|\hat p- \hat k|^{2+\alpha(x)}} = \frac{1}{2^{1+\alpha(x)/2}} \rho(x,\hat k \cdot \hat p), \qquad\text{with}\qquad \rho(x,s) := \frac{a\big(\sqrt{2(1-s)}\big)}{(1 - s )^{1+\alpha(x)/2}}\qquad s\in [-1,1).
\end{equation}
Above, $\alpha :\R^3 \longrightarrow [0,2)$ accounts for the slow variations of scattering across the ambient space, and $a$ is a smooth bounded function characterizing some statistical properties of the medium and such that $a(0)>0$. Practical examples are given further. A direct calculation shows that \fref{eq:mft} holds when $\alpha \in [0,2)$. Also, the integral in \fref{defQ} has to be understood in the principal value sense when $\alpha \in [1,2)$, see \cite{gomez2}. The multifractional terminology that we use is motivated by the fact that the unbounded operator $\calQ$ can be expressed as a (multi)-fractional Laplace-Beltrami operator $(-\Delta_{\bbS^2})^{\alpha(x)/2}$ on the unit sphere up to a bounded operator w.r.t. the $\hat k$ variable \cite{gomez1, gomez2}.

We would like to emphasize that we focus in this work on kernels of the form \fref{def:rho} for simplicity of the exposition, and that our method applies, after proper decomposition (see \cite{gomez2}), to more general kernels that behave like \fref{def:rho} at the singularity.

The RTE can be derived from high frequency wave propagation in random media, see e.g. \cite{Ryzhik}. In such a context, the velocity field $c(x)$ has the form
\[
\frac{1}{c^2(x)} = \frac{1}{c^2_0}\Big( 1 +  \sqrt{\eta} \, V_0\Big(x , \frac{x}{\eta}\Big) \Big)\qquad x\in\mathbb{R}^3,\quad\eta \ll 1,
\]
where $c_0$ is the background velocity (that we set to one in the sequel for simplicity), $V_0$ is a mean zero random field modeling fluctuations around the background, and $\eta$ is the correlation length of the random medium, assumed to be small after proper rescaling. The first variable in $V_0$ represents the slow variations of the random perturbations, while the second one corresponds to their high frequency oscillations. The latter are responsible for the strong interaction between the wave and the medium over sufficient distances. The scattering kernel $\Phi$ is related to the correlation function of $V_0$, and assuming $V_0$ is stationary (in the statistical sense) with respect to the fast variable, a kernel of the form (\ref{def:rho}) can be obtained from random fields such that
\begin{equation}\label{eq:cov_media}
\E[V_0(x, x')V_0(y,y')] = \sqrt{\lambda(x)\lambda(y)} \int_{\R^3} \frac{a(|p|)}{|p|^{1+\frac{\alpha(x) + \alpha(y)}{2}}} e^{ip\cdot (x'-y')}dp,
\end{equation}
with $\alpha$ ranging from $0$ to $2$. Denoting by $R(x)$ the expectation in \fref{eq:cov_media} with $y=x$, $y'=x'+x/\eta$, one can show that $R$ behaves like $|x|^{\alpha(x)-2}$ for $|x| \gg 1$, and is therefore not integrable. This is how random fields with long-range correlations are defined, as opposed to random fields with short-range correlations that exhibit an integrable correlation function. This approach is of practical interest in biomedical imaging as media with long-range correlations are able to reproduce experimentally observed power-law attenuations associated with effective fractional wave equations \cite{garnier, gomez}. The value of the exponents is related to the rate of decay of the correlation function $R$, and depends on the nature of the imaged tissues as reported in \cite{duck, goss78, goss80}. Variations of this exponent can then be used for diagnosis purposes \cite{lin, rau}.

\begin{figure}[h!]
\begin{center}
\includegraphics[scale=0.3]{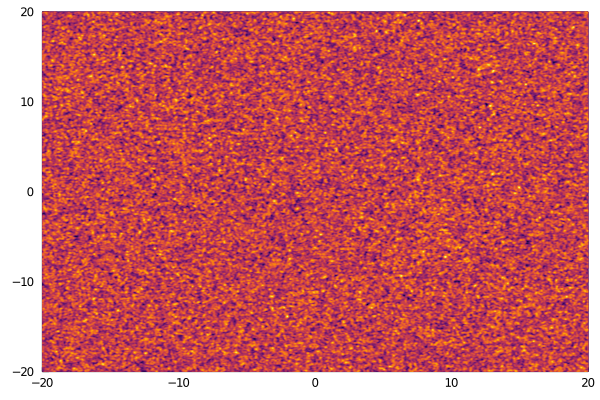} 
\includegraphics[scale=0.3]{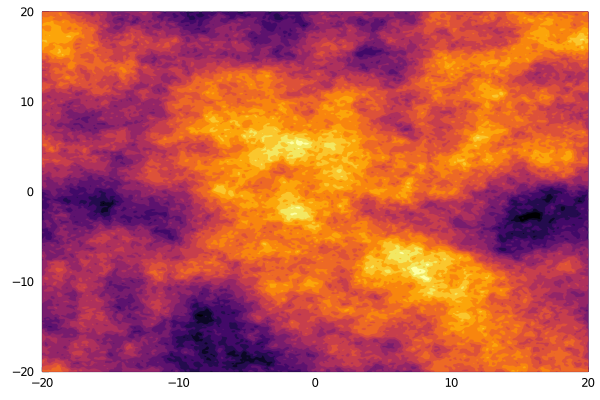} 
\\

\includegraphics[scale=0.3]{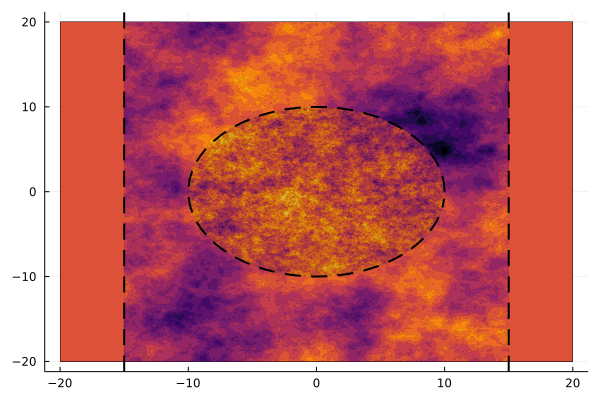} 
\includegraphics[scale=0.3]{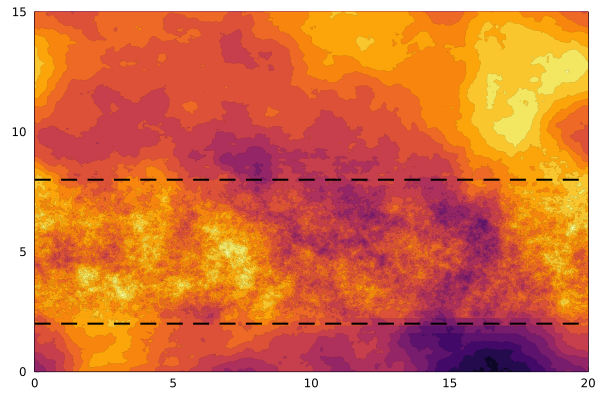} 
\end{center}
\caption{\label{fig:random_media} Realizations of Gaussian random fields. The upper-left picture represents a field with short-range correlations, while the upper-right picture depicts a field with long-range correlations with $\lambda \equiv 1$ and $\alpha \equiv 1$. In the lower-left picture, we have $\lambda = \mathbf{1}_{\{|x_1|<15\}}$, and $\alpha=0.1 \cdot \mathbf{1}_{\{|x|\leq 10\}} + 1 \cdot \mathbf{1}_{\{10<|x|\}}$. In the lower-right, we have $\lambda \equiv 1 $ and $\alpha(x_2) = 5/3 \cdot \mathbf{1}_{\{x_2\leq 2\}} + 0.5 \cdot \mathbf{1}_{\{2<x_2\leq 8\}}+ 1.9\cdot \mathbf{1}_{\{8<x_2\}}$.}
\end{figure}
In Figure \ref{fig:random_media}, we provide  examples of such 2D random fields. The top-left picture represents a random medium with short-range correlations (with a standard Gaussian covariance kernel), while the top-right picture illustrates a random medium with long-range correlations with $\alpha\equiv 1$. Because of the singularity at $p=0$, one can observe significantly larger statistical patterns than in the short-range case. 
In the bottom two pictures, we highlight the roles of $\lambda$ and $\alpha$: $\lambda$ characterizes scattering regions, and $\alpha$ defines the correlation structure. In the inner circle of the bottom-left picture we have $\alpha\equiv 0.1$, which tends to create shorter range fluctuations than in the outside where $\alpha\equiv 1$. In the bottom-right picture, we have a three-layer model for $\alpha$ in which the inner band exhibits smaller statistical patterns than the outer ones. This type of model is used for modeling non-Kolmogorov atmospheric turbulences, while standard atmospheric turbulence is modeled with the so-called Kolmogorov power spectrum 
\[
\Phi(|k|) \propto \frac{a(|k|)}{|k|^{11/3}}, 
\]
for $|k|$ in the inertial range of turbulence. This corresponds to the case $\alpha = 5/3$. This case is not always valid in experiments as reported in \cite{belenkii, stribling, zilberman}, and the statistics of atmospheric turbulence have been shown to vary with altitude. Models have been derived for instance (see \cite{korotkova} for a review) by considering three ranges (0-2km, 2-8km, and above 8km) with distinct power laws (see Figure \ref{fig:turb} for an illustration).

In the context of biological tissues, the following the Gegenbauer scattering kernel $\rho_G$ and Henyey-Greenstein (HG) kernel $\rho_{HG} $ are commonly used in the peaked-forward regime \cite{Henyey, Reynolds}:
\be \label{HG}
\rho_G(x,s) := \frac{\alpha \,g \, (1 + g^2 - 2g \, s)^{-1-\alpha/2} }{2\pi ( (1-g)^{-\alpha} - (1+g)^{-\alpha})},\qquad \rho_{HG}(x,s) := \frac{1}{4\pi} \frac{1-g^2}{(1+g^2 - 2g \, s)^{3/2}}.
\ee
The parameter $g\in(-1,1)$ is called the anisotropy factor, and $\rho_{HG}$ is obtained by setting $\alpha\equiv 1$ in $\rho_G$. The case $g=0$ corresponds to isotropic energy transfer over the unit sphere, $g<0$ to dominant transfer in the backward direction, and $g>0$ to forward energy transfer. The peaked forward regime is obtained in the limit $g\to 1$, for which 
\begin{equation}\label{eq:lim_HG}
\frac{1}{(1-g)^\alpha}\rho_G(x,\hat k \cdot \hat p) \underset{g\to 1}{\sim} \frac{\alpha}{2\pi(2-2\hat k \cdot \hat p)^{1+\alpha/2}} = \frac{\alpha}{2\pi |\hat k - \hat p|^{2+\alpha}}.
\end{equation}

The case $\alpha\equiv 1$ for the HG kernel is widely used in photon scattering in biological tissues \cite{Dehaes, Germer, Jacques}. A typical realization of the corresponding random field in 2D as $g\to 1$ is depicted in the top-right panel of Figure \ref{fig:random_media}. 

There exist a variety of methods for the resolution of \fref{RTE} that handle the singular nature of the HG kernel, see e.g. \cite{Fujii, zhao, Kim, Kim2, Leakeas}. They are based on finite differences type discretizations, projections over appropriate bases w.r.t. the $\hat{k}$ variable, and approximations of the kernel. Here we propose an alternative approach to handle singular scattering kernel \eqref{def:rho} that is based on a MC method. The latter are popular choices for the simulation of the RTE when the kernel is smooth, see e.g. \cite{Lapeyre,margerin, margerin1, Przybilla, Spanier}, essentially for their adaptability to a wide range of configurations and their simplicity of implementation. A downside is their slow convergence rate, and there is a vast literature on variance reduction techniques for acceleration. In this work, we focus on the design of an efficient MC method and postpone any variance reduction considerations to future works.



Our approach is based on an adaptation of a method proposed by Asmussen-Cohen-Rosi\'nski \cite{Asmussen, Cohen} (ACR) for the simulation of L\'evy processes with infinite jump intensity. It relies on a \emph{small jumps/large jumps} decomposition of the corresponding infinitesimal generator. The main idea is to approximate the generator of the \emph{small-jump} part, which possesses the infinite intensity due to the singularity of kernel, by a Laplace-Beltrami operator (with respect to the angular variables) on the unit sphere $\mathbb{S}^2$.  This requires us to simulate paths of a jump-diffusion process over the unit sphere. For this purpose, we use the characterization of Brownian motion on the unit sphere given in \cite{Berg} based on a standard stochastic differential equation (SDE) in $\mathbb{R}^3$ that is suitable for space-dependent kernels. This situation is hence more involved than the 2D case we investigated in \cite{gomez3} where the small jumps part can be approximated by Brownian motion on the unit circle for which analytical expressions are available. Note that, as shown in \cite{gomez3}, neglecting small jumps altogether in order to use standard MC methods leads to large errors, and reducing those comes at significantly increased computational cost.


Denoting by $\hat \mu(u)$ the estimator produced by our MC method for some observable $\mu(u)$ built on the solution $u$ to \eqref{RTE}, we provide an error estimate of the form
\[\mathbb{P}\Big(|\hat \mu(u)-\mu(u)|> E_1 + E_2 + E_3\Big) \ll 1\]
as a theoretical support of our method. Above, $E_1$, $E_2$, and $E_3$ are small terms characterizing the various approximation errors from the original model: the Laplace-Beltrami (i.e. small jumps) approximation, the discretization error of the diffusion process over the unit sphere, and the MC error. 
Note that the method we propose here applies directly to the stationary version of \fref{RTE} 
\[
\hat k \cdot \nabla_x \mathbf{u}  - \calQ \mathbf{u} = u_0,\qquad (x,\hat k) \in \mathbb{R}^3 \times \mathbb{S}^2,
\]
with source term $u_0$, through the relation 
\[
\mathbf{u} (x,\hat k) := \int_0^\infty u(t, x, \hat k) dt.
\]

The paper is organized as follows. In Section \ref{sec:prob_rep}, we introduce probabilistic representations for \eqref{RTE} and its approximation based on the ACR method. In Section \ref{sec:MC_meth}, we describe our MC method, state the main theoretical result regarding the overall approximation error, and detail the simulation algorithms. Section \ref{sec:test_case} is dedicated to the validation of the method using semi-analytical solutions. Numerical illustrations are given in Section \ref{sec:numilla}, where we investigate the role of the strength $\alpha$ of the singularity, both when constant or space-dependent in the case of non-Kolmogorov turbulence, and compare with solutions for the HG kernel. Section \ref{secproofs} is devoted to the proofs of our main results and we recall in an Appendix the stochastic collocation method.

The numerical simulations are performed using the Julia programming language (v1.6.5) on a NVIDIA Quadro RTX 6000 GPU driven by a 24 Intel Xeon Sliver 2.20GHz CPUs station. The codes have been implemented using the CUDA.jl library \cite{besard1, besard2}. 

\paragraph{Acknowledgment.} OP acknowledges support from NSF grant DMS-2006416. 

\section{Probabilistic representations and approximation}\label{sec:prob_rep}

\subsection{Representation for \fref{RTE}}
The starting point is the following standard probabilistic interpretation to \eqref{RTE}:
  \[
u(t,x,k) = \E_{x,\hat k}\big[u_{0}( D(t) )  \big]:=\E\big[u_{0}( D(t) )\,|\, D(0) = (x,\hat k) \big],
\]
where $D = (X, K)$ is a Markov process on $\mathbb{R}^3 \times \mathbb{S}^2$ with infinitesimal generator
\[
\calL f(x, \hat k) := -\hat k \cdot \nabla_x f(x, \hat k) + \frac{\lambda(x)}{2^{1+\alpha(x)/2}}\int_{\bbS^2} \rho(x, \hat p \cdot \hat k) \big(f(x, \hat p)-f(x, \hat k)\big)\sigma(d\hat p).
\]
A path, or a realization, of the Markov process $D$ is often referred to as a \emph{particle trajectory}. The $X$ component of $D$ represents the position of a particle, and the component $K$ its direction. The generator $\calL$ comprises two terms, the transport part describing free propagation of the particle, and the scattering operator (often referred to as the jump part in the probabilistic literature) describing the evolution of its direction. The jump component exhibits a non-integrable singularity leading to a infinite jump intensity and a vanishing mean free time as expressed in \fref{eq:mft}.

Note that when $\lambda$ and $\alpha$ are constant, it is shown in \cite{gomez2} that the solution $u$ is unique and infinitely differentiable in all variables for $t>0$ for any square integrable initial condition. When $\lambda$ and $\alpha$ are infinitely differentiable with bounded derivatives at all orders, this result remains valid and we will assume throughout this work that $u$ is smooth. The same applies to the function $u_\eps$ defined further in Proposition \ref{prop:approx_RTE}.

In order to adapt the ACR method, we introduce the following small region over which the singularity of the kernel $\rho$ (in \eqref{def:rho}) is not integrable, resulting in an unbounded infinitesimal generator $\calL$:
\begin{equation}\label{def:Seps}
S^\eps_{<} = S^\eps_{<}(\hat k):=\{\hat p \in \bbS^2:\quad 1 - \hat p\cdot\hat k < \eps\}\qquad \eps\in(0,1).
\end{equation}
We can now decompose the jump part of the generator $\calL$ into two components
\[\begin{split}
\calL f(x,\hat k) &= - \hat k \cdot \nabla_x f(x,\hat k) + \calL^\eps_{<} f(x,\hat k) + \calL^\eps_{>} f(x,\hat k)\\
& := - \hat k \cdot \nabla_x f(x,\hat k) + \frac{\lambda(x)}{2^{1+\alpha(x)/2}}\left(\int_{S^\eps_{<}} + \int_{S^\eps_{>}} \right)\rho(x, \hat p \cdot \hat k)\big(f(x, \hat p)-f(x, \hat k)\big)\sigma(d\hat p),
\end{split}\]
where $S^\eps_{>} = (S^\eps_{<})^c$ is the complementary set of region \eqref{def:Seps} over the unit sphere. The part of the scattering operator involving $S^\eps_{>}$ (with no singularity) is the infinitesimal generator of a standard jump Markov process. Regarding $S^\eps_{<}$ (with the singularity), the following result justifies the approximation of this singular part by a Laplace-Beltrami operator $\Delta_{\mathbb{S}^2}$ over the unit sphere $\mathbb{S}^2$. 
We will use the notation $r'_\eps=\sqrt{1-(1-\eps)^2}/(2-\eps)$ in what follows, and set in the rest of the paper $0<\eps\leq \eps_0<1$ and $0 \leq \alpha_m\leq \alpha(x) \leq \alpha_M<2$. 

\begin{prop}\label{prop:approx_RTE} Let $u$ be the solution to \eqref{RTE} and $u_\eps$ be the solution to  
\begin{equation}\label{RTE_approx}
\left\{ \begin{split}
\partial_t u_\eps + \hat k \cdot \nabla_x u_\eps  & = \sigma^2_{\eps}(x) \Delta_{\mathbb{S}^2} u_\eps + \frac{\lambda(x)}{2^{1+\alpha(x)/2}} \int_{S^\eps_{>}} \rho(x, \hat p \cdot \hat k) (u_\eps(\hat p)-u_\eps(\hat k))\sigma(d \hat p),\\
u_\eps(0,x,\hat k) & = u_{0}(x,\hat k),
\end{split}\right.
\end{equation}
for $(t,x,\hat k)\in (0,\infty)\times \R^3\times \mathbb{S}^2$, where
\begin{equation}\label{def_sig}
\sigma^2_{\eps}(x) := \frac{2^{1-\alpha(x)} a(0) \pi \lambda(x)}{2-\alpha(x)} {r'_\eps}^{2-\alpha(x)}.
\end{equation}
Assuming $a'(0)=0$, for any $T>0$, we have
\begin{equation}\label{eq:err_bound}
 \sup_{t\in [0,T]}\| u(t,\cdot,\cdot) - u_\eps(t,\cdot,\cdot) \|_{L^2(\mathbb{R}^3 \times \mathbb{S}^2)}  \leq  \eps^{2-(\alpha_M/2)} \, \sqrt{2T}\, E(u)
\end{equation}
where $E(u)$ is defined in \fref{eq:def_C_eps}.
\end{prop} 

The proof of Proposition \ref{prop:approx_RTE} is postponed to Section \ref{proof:prop_approx_RTE}. The term $E(u)$ is independent of $\eps$ and depends on derivatives of $u$ w.r.t. $\hat k$ up to order 4. Note that the error is of order $\eps^{1-(\alpha_M/2)}/(2-\alpha_M)$ when $a'(0)\neq 0$ yielding a less accurate approximation than for $a'(0)=0$. The difference comes from a truncated expansion along the sphere curvature providing an extra order in $\eps$ assuming $a'(0)=0$. This later assumption holds throughout the remaining of the paper. Based on \fref{eq:err_bound}, we then devise a MC method for \fref{RTE_approx} instead of \fref{RTE}. The advantage in using \fref{RTE_approx} is the fact that the  angular diffusion term $\sigma^2_{\eps}(x) \Delta_{\mathbb{S}^2}$ is the generator of a Markov process that can be easily simulated. Indeed, for $W$ a standard 3D Brownian motion on $\mathbb{R}^3$ and $\times$ the cross product in $\mathbb{R}^3$, it is shown in \cite{Berg} that the process $B$ solving the SDE
\[
dB = B \times dW - B dt, \qquad B(0)\in \mathbb{S}^2,
\]
has generator $\frac{1}{2} \Delta_{\mathbb{S}^2}$. A simple adaptation then gives the desired diffusion coefficient. Since the error is of order $\eps^{2-(\alpha_M/2)}$, it is always smaller than $\eps$, and can be adjusted to obtain a desired accuracy. Note also that $\sigma_{\eps}(x)$ increases as $\alpha(x)$ gets to $2$, and diffusion on the sphere eventually becomes the dominant dynamics.

\subsection{Representation for \eqref{RTE_approx}}

We interpret \fref{RTE_approx} as the forward Kolmogorov equation of an appropriate Markov process, and as a consequence focus on \textit{forward} MC methods, see e.g. \cite{Lapeyre} for terminology. Backward equations are simulated in a similar manner, and can be combined with forward methods for variance reduction techniques \cite{Blanc, Lux, Spanier}.

The Markov process we consider for this approach is defined by
\begin{equation}\label{def_Z}
D_\eps(t) := \sum_{n\geq 0} \mathbf{1}_{[T_n, T_{n+1})}(t) \psi^{Z_{n}}_n(t-T_n)\qquad t\geq 0,
\end{equation}
where (in the remaining of the paper we extensively make use of the notation $z=(x,\hat k)$):
\begin{enumerate}
\item The flow $\psi^{z}_n = (X^{x}_{n}, K^{\hat k}_{n})$ is the unique strong solution to the SDE 
\begin{equation}\label{diff_eq}
\left\{ \begin{array}{rcl}
dX^{x}_{n}(t) & = & K^{\hat k}_{n}(t)\, dt \\ 
dK^{\hat k}_{n}(t) & = & \sqrt{2}\, \sigma_{\eps}(X^{x}_{n}(t))\, K^{\hat k}_{n}(t) \times dW_n(t) - 2\,\sigma_{\eps}^2(X^{x}_{n}(t))\, K^{\hat k}_{n}(t) \, dt,
\end{array}
\right.
\end{equation}
where $\psi^{z}_n(0) = z$, $\times$ is the cross product in $\mathbb{R}^3$, $(W_n)_n$ is a sequence of independent standard Brownian motions on $\mathbb{R}^3$, and $\sigma_{\eps}$ is defined by \eqref{def_sig}.
\item The jump times $(T_n)_n$ are distributed according to
\[
\mathbb{P}\big(T_{n+1} -T_{n} > t\,| \, D_\eps(T_{n}) = z,\, (\psi^{z}_n(s))_{s\in[0,t]}\big) = \exp\Big( - \int_0^t \Lambda_\eps(\psi^{z}_n(s)) ds   \Big), \qquad \forall n \geq 0,
\]
with $T_0=0$, and for $\rho$ given by \eqref{def:rho},
\begin{equation}\label{eq:Lambda}
 \Lambda_\eps(z) :=  \frac{\lambda(x)}{2^{1+\alpha(x)/2}} \int_{S^\eps_{>}} \rho(x, \hat p \cdot \hat k ) \sigma(d\hat p).
\end{equation}

\item The jumps $(Z_n)_n$ describe a Markov chain with transition probability
\begin{equation}\label{def_jump}
\mathbb{P}(Z_{n+1}\in dy \otimes \sigma(d\hat p) \,\vert Z_n, T_{n+1} - T_n ) = \Pi_\eps( z_{n+1} , \, dz ),
\end{equation}
where $z_{n+1} := \psi^{Z_{n}}_n(T_{n+1}-T_n)$, and
\begin{equation}\label{def_Pi}
\Pi_\eps(z, dz) := \pi_\eps(z, \hat p) \sigma(d\hat p)\delta_{x}(dy),
\end{equation}
with density
\begin{equation}\label{def_pi}
\pi_\eps(z, \hat p) := \frac{\rho(x, \hat p \cdot \hat k )}{\int_{S^\eps_{>}} \rho(x, \hat p' \cdot \hat k )d\sigma(\hat p') } \mathbf{1}_{S^\eps_{>}}(\hat p), \qquad z=(x,\hat k),
\end{equation}
which is supported over $S^\eps_{>}$. The above Dirac mass $\delta_x(dy) := \delta(x-y)dy$ translates the fact that the jumps only hold w.r.t. the $\hat k$ variable.
\end{enumerate}

Let us note that the above family of standard Brownian motions $(W_n(t))_{t\in[0, T_{n+1}-T_n]}$ can be defined as
\[
W_n(t) = W(t+T_n) - W(T_n),\qquad t\in[0, T_{n+1}-T_n],
\]
for any $n$, where $W$ is a single standard Brownian motion on $\mathbb{R}^3$. We have then the following probabilistic representation for the solution to \eqref{RTE_approx}.

\begin{prop}\label{rep_prob2} The Markov process $D_\eps$ defined in \fref{def_Z} has for infinitesimal generator
  \[
    \calA_\eps g(z)  := \hat k \cdot \nabla_x g(z) + \sigma^2_{\eps}(x)\Delta_{\mathbb{S}^2} g(z) + \Lambda_\eps(z) \int_{\mathbb{S}^{2}} \pi_\eps(z, \hat p) \big(g(x, \hat p)-g(x, \hat k)\big)\sigma(d\hat p),
    \]
and we have
\begin{equation}\label{eq:prob_rep2}
\mathbb{P}_{\mu_0}( D_\eps(t) \in dx\otimes \sigma(d\hat k) ) = \frac{1}{\bar u_0} u_\eps(t, x, \hat k)\,dx\, \sigma(d\hat k),
\end{equation}
where 
\begin{equation}\label{eq:def_mu0}
 \mu_0(dx, d\hat k) := \mathbb{P}( D_\eps(0) \in dx\otimes \sigma(d\hat k) )= \frac{u_{0}(x,\hat k) }{\bar u_0}\,dx\, \sigma(d\hat k) \qquad\text{with}\qquad \bar u_0 := \int_{\mathbb{R}^3\times \mathbb{S}^2} u_{0}(x,\hat k) \, dx \,\sigma(d\hat k).
 \end{equation}
\end{prop}

The terminology \emph{forward} comes from the fact that the particles are emitted at random points at time $t=0$ (through $\mu_0$) and propagate towards the observation position $z=(x,\hat k)$. The proof of Proposition \ref{rep_prob2} is provided in Section \ref{proof_rep_prob2}. Let us illustrate two aspects of the representation \eqref{eq:prob_rep2}. In order to obtain an estimation of $u_\eps(t, x, \hat k)$ at the point $z=(x,\hat k)$, we calculate the probability
\begin{equation}\label{eq:energy_approx}
\frac{\bar u_0}{|\mathbf{B}(z,r)|} \, \mathbb{P}_{\mu_0}\big( D_\eps(t) \in \mathbf{B}(z,r) \big) \simeq u_\eps(t, x, \hat k),
\end{equation}
where $\mathbf{B}(z,r)\subset \mathbb{R}^3\times \mathbb{S}^2$ stands for the open ball centered at $z=(x,\hat k)$ with radius $r \ll 1$. If we are only interested in e.g. the energy density at point $x$, we estimate
\begin{equation*}
\frac{\bar u_0}{|B(x,r)|} \, \mathbb{P}_{\mu_0}\big(D_{\eps}(t) \in B(x,r)\times \mathbb{S}^2\big) = \frac{\bar u_0}{|B(x,r)|} \mathbb{P}_{\mu_0}\big( D_{1, \eps}(t) \in B(x,r)\big)  \simeq \int_{\mathbb{S}^2} u_\eps(t, x, \hat k) \sigma(d\hat k),
\end{equation*}
where $B(x,r) \subset \mathbb{R}^3$ stands for the open ball centered at $x$ with radius $r \ll 1$, and $D_{1, \eps}$ is the $x$ component of $D_\eps$.

\section{Monte Carlo Method}\label{sec:MC_meth}

Based on the previous probabilistic representation of \fref{RTE_approx}, solving \eqref{RTE_approx} requires the generation of random paths of the stochastic process $D_\eps$. For any measurable bounded functions $f$, the convergence of the estimator
\[
\mu_N(t,f) := \frac{1}{N} \sum_{j=1}^N f(D^j_\eps(t)) \underset{N\to\infty}{\longrightarrow}  \int f(x,\hat k) u_\eps(t,x,\hat k) \, dx \, \sigma(d \hat k)\qquad \mathbb{P}_{\mu_0}-\text{almost surely},
\]
is guaranteed by the strong law of large numbers.  Above $(D^{j}_\eps)_j$ is a sample of $D_\eps$. We detail next how to treat efficiently the diffusion and jump components of the process $D_\eps$.

\subsection{The jump part}

Since the process $D_\eps$ is inhomogeneous, i.e. $\Lambda_\eps$ and $\Pi_\eps$ both depend on $z=(x, \hat k)$, we use the so-called \emph{thinning} method, also referred to as the \emph{fictitious shocks} method \cite{Lapeyre}. It is based on a acceptation/rejection step and consists in simulating at first more jumps (or shocks) than necessary. In a second step, some of the jumps are rejected according to an appropriate probability distribution in order to recover the original dynamics. Assume $0<\alpha_m\leq \alpha(x) \leq \alpha_M<2$. A direct calculation shows that 
\begin{equation}\label{eq:def_bar_Lambda}
\Lambda_\eps(z) \leq \frac{2\pi \sup \lambda \sup a}{2^{1+\alpha(x)/2}}\int_{-1}^{1-\eps}\frac{dt}{(1-t)^{1+\alpha(x)/2}}=\frac{4\pi \sup \lambda \sup a}{\alpha(x) 2^{1+\alpha(x)/2} \eps^{\alpha(x)/2}} \leq \frac{2\pi \sup \lambda \sup a}{\alpha_m \eps^{\alpha_M/2}} =: \bar \Lambda_\eps.
\end{equation} 
The fictitious jump times are then drawn as 
\[\bar T_n :=\sum_{j=1}^n \xi_j,\qquad\text{and}\qquad \bar T_0 = 0,\]
where the $(\xi_j)_j$ are i.i.d. exponentially distributed random variables with parameter $\bar \Lambda_\eps$.

The thinning method consists in the following acceptation/rejection step. At a jump time $\bar T_n$ and current position $z_n$, we draw a jump $z$ according to the probability distribution $\Pi_\eps$ given by \eqref{def_jump}. This jump is accepted with probability $p(z_n) = \Lambda_\eps(z_n) /\bar \Lambda_\eps$. Otherwise, the process $D_\eps$ continues to diffuse starting from $z_n$, and $\bar T_n$ is not considered as a \emph{true} jump time for $D_\eps$. Practically, we can define the state as 
\[
\bar Z_n := z \, \mathbf{1}_{(U_n \leq p(z_n))} +  z_n \, \mathbf{1}_{(U_n > p(z_n))}  
\] 
at each fictitious jump times $\bar T_n$. Above, $U_n$ is a random variable uniformly distributed over $[0,1]$ and all the $U_n$'s are independent.

\subsection{The diffusion part} \label{secdifpart}

The diffusion part between two jumps satisfies the linear SDE \fref{diff_eq}, and is simulated using the following Euler-Maruyama type scheme 
\begin{equation}\label{def:scheme}
(S_{n,m}):\left\{ \begin{array}{rcl}
X_{n, m+1} & = & X_{n, m} + h_{n,m} \, \hat K_{n, m}  \\
K_{n, m+1} & = & \hat K_{n,m} - 2h_{n,m} \, \sigma^2_{\eps}(X_{n, m} )\, \hat K_{n,m} +  \sqrt{2 h_{n,m}} \, \sigma_{\eps}(X_{n, m}) \, \hat K_{n, m} \times W_{n,m}\\
\hat K_{n,m+1} & =& \frac{K_{n,m+1}}{|K_{n,m+1}|},
\end{array} 
\right.
\end{equation}
where the $(W_{n,m})_{m,n}$ are i.i.d. mean-zero Gaussian random vectors with identity covariance matrix. Note that the above scheme does not conserve the Euclidean norm with respect to the angular variable, and as a consequence the evolution of $(K_{n,m})_{n,m}$ does not remain on the unit sphere over the iterations. This motivates the definition of $\hat K_{n,m}$. We have nevertheless $
\mathbb{E}[\,|K^{k}_{n, m}|^2\,]=1$, for all $n$ and $m$, 
and Theorem \ref{mainth} below guarantees that the distribution of $\hat K_{n,m}$ provides a converging approximation of the true statistics.  The stepsizes $h_{n,m}$ are determined from a fixed stepsize $h$ as follows. Since our convergence theorem further is stated at a fixed time $T$, we include $t=T$ in the discretization grid for simplicity. Let then $N_T$ such that $\bar T_{N_T} \leq T <\bar T_{N_T+1}$. When $n \neq N_T$, let $dt_n=\bar T_{n+1}-\bar T_n$. With $m_n = [dt_n/h]$ ($[\cdot]$ the integer part), we set, for $n \neq N_T$,
\[
h_{n,m} = \left\{ \begin{array}{ccc} 
h & for & m=0,\dots,m_n-1 \\
dt_n - m_n h & for & m=m_n,
\end{array}\right.
\]
and if $m_n=0$, we set $h_{n,0}=dt_n$. In the rest of the paper, the grid is denoted by $(T_{n,m})$, where $T_{n,m+1} =T_{n,m} + h_{n,m}$ for $T_{n,0}=\bar T_n$, $n \geq 0$ and $m=0,\dots,m_{n}-1$. When $n=N_T$, we divide the interval $[\bar T_{N_T},\bar T_{N_T+1}]$ similarly into subintervals of length $h_{N_T,m}$ at most $h$ (we suppose there are $m_{N_T}$ of those) and such that $T_{N_T,m}=T$ for one $m$ in $0,\dots,m_{N_T}$. 

\subsection{The overall discretized process and convergence}

For any $t \geq 0$, the approximate version of the  process $D_\eps$, denoted  $D_{h,\eps}$,  is defined by:
\[ D_{h,\eps} (t) = (X_{h,\eps}(t), K_{h,\eps}(t)) := \sum_{n = 0}^{\infty} \sum_{m=0}^{m_n} \mathbf{1}_{[T_{n,m},T_{n,m}+h_{n,m})}(t)Z_{n,m},
\]
where
\begin{enumerate}
\item For any $m\in \{0,\dots,m_n\}$,
\[
Z_{n,m+1} = S_{n,m}(Z_{n,m}),
\]
where $Z_{n,0}=\bar Z_n$ for the $(\bar Z_n)_n$ defined below, and where $S_{n,m}(Z_{n,m})=(X_{n,m+1},\hat K_{n,m+1})$ is given by the scheme \eqref{def:scheme} with initial condition $Z_{n,m}=(X_{n,m},\hat K_{n,m})$. 
\item The sequence $(\bar Z_n)_n$, is defined by
    \[
      \bar Z_{n+1} = z \, \mathbf{1}_{(U_{n} \leq p(Z_{n,m_n+1}))} +  Z_{n,m_n+1}  \, \mathbf{1}_{(U_{n} > p(Z_{n,m_n+1}))}\qquad n\geq 0,
      \]
where $z$ is drawn according to the probability measure $\Pi_\eps( Z_{n,m_n+1} , dz)$ defined by \eqref{def_Pi}.
\end{enumerate}

Below, $b$ is the backward solution to \fref{RTE} with terminal condition $b(T,x,\hat k) =  f(x,\hat k)$, see \fref{eq:back_RTE}. Our convergence result is then the following (we set $h$ such that $4 h \sup_x \sigma_\eps^2(x) \leq 1$ to simplify some expressions):

\begin{thm} \label{mainth}
Consider
\[
\mu_{N, h, \eps}(t,f) = \frac{1}{N} \sum_{j=1}^N f( D^j_{h,\eps}(t)),\qquad \mu(t,f) =\int f(x,\hat k) u(t,x,\hat k) \, dx \, \sigma(d \hat k),
\] 
where $(D^j_{h,\eps})_j$ is a sample of $D_{h,\eps}$. For any $T>0$,  $\eta >0$ and any smooth bounded function $f$ on $\mathbb{R}^3\times \mathbb{S}^2$, we have
\begin{equation}\label{eq:main_ineq} 
\limsup_{N\to \infty}\mathbb{P}\Big(|\mu_{N, h, \eps}(T,f)-\mu(T,f)|> \frac{\eta \Sigma_{h,\eps}}{ \sqrt{N}} + \eps^{2-(\alpha_M/2)}F_0(u,b,f) + h F_1(b)\Big) \leq \emph{erfc}(\eta/\sqrt{2}),
\end{equation}
where
\[\Sigma_{h,\eps} = \sqrt{Var\big(f(D_{h,\eps}(T))\big)} \leq \sup |f|.\]
The functions $F_0$ and $F_1$ are explicit and independent of $\eps$ and $h$, and are defined in the proof of the theorem in Section \ref{proof_cv}.
\end{thm}
Theorem \ref{mainth} is proved in Section \ref{proof_cv}. In \eqref{eq:main_ineq}, there are three terms that quantify the approximation error of our estimator $\mu_{N, h, \eps}(t,f)$: one of order $\eps^{2-(\alpha_M/2)}$ due to the approximation of $u$ by $u_\eps$ (the smaller the $\alpha_M$, i.e. the less singular the kernel is, the smaller the error), one of order $h$ due to the numerical approximation of the diffusion over the unit sphere, and one due to the MC approximation with the standard $1/\sqrt{N}$ convergence rate. Note that the discretization error of the diffusion process is only of order $h$ and not of order the standard $\sqrt{h}$. The reason is that we are only interested in the convergence of Monte Carlo estimators, allowing us to consider this discretization error in the weak sense \cite{talay}. However, a weak second-order Runge-Kutta method can be considered to provide an error in $h^2$ instead of $h$ for the Euler scheme \cite{debrabant}. Modifications of the SDE \eqref{diff_eq} can also be considered to provide weak higher-order scheme \cite{abdulle}. The main goal of this paper being to present a methodology to capture efficiently the behavior induced by the singularity, we focus our attention on the error in $\eps$, and do not present weak-higher order discretization schemes for the SDE. In this way, the Euler scheme is considered for simplicity in the proof of Theorem \ref{mainth}. For the numerical simulations of Sections \ref{sec:test_case} and \ref{sec:numilla}, that illustrate the roles of $\varepsilon$ and $\alpha$ in the approximation, the parameter $h$ will be chosen proportionally to the shortest mean free time $\bar \Lambda_\eps^{-1}$, and small enough so that the approximation error w.r.t. $\eps$ in dominant. $N$ will also be chosen large enough so that the error of approximation in $\eps$ is dominant. The MC error is controlled by the standard deviation $\Sigma_{h,\eps}$, and variance reduction techniques can be designed to reduce this term. When estimating the energy density over a given region $B$, as in \eqref{eq:energy_approx}, the number of particles $N$ needed to reach a given error threshold can be estimated as follows: the root mean square error of the MC estimator for $f=\mathbf{1}_B/|B|$ reads
\begin{equation}\label{def:RMSE}
RMSE_{h,\eps} := \frac{\E\big[\big(\mu_{N, h, \eps}(T,f) - \E\big[f(D_{h,\eps}(T))]\big)^2\big]^{1/2}}{\E\big[f(D_{h,\eps}(T))\big]} = \sqrt{\frac{P_{h,\eps}(1-P_{h,\eps})}{N |B|^2}}\leq \frac{1}{2 \sqrt{N} |B|},
\end{equation}
and the relative MC error is
\begin{equation}\label{def:rel_err}
\frac{RMSE_{h,\eps}}{\E\big[f(D_{h,\eps}(T))\big]} = \frac{1}{\sqrt{N}} \sqrt{\frac{1-P_{h,\eps}}{P_{h,\eps}}}\leq \frac{1}{\sqrt{N P_{h,\eps}}},
\end{equation}
with
\[
P_{h,\eps} := \Pro_{\mu_0}(D_{h,\eps}(T)\in B) \simeq \frac{1}{\bar u_0 }\int_B u_{\eps}(t,x,k) dx \sigma(d\hat k).
\]
Above, $\mu_0$ and $\bar u_0$ are given by \eqref{eq:def_mu0}. A RMSE lower than a threshold $c$ would then require
\begin{equation}\label{eq:RMSE}
N\geq \frac{1}{4c^2|B|^2},
\end{equation}
while a relative error would require
\begin{equation}\label{eq:rel_err}
N \geq \frac{1}{c^2 P_{h,\eps}}.
\end{equation}
If $B$ is a region centered around a point $(x_0, k_0)$, with a small volume (that is $ P_{h,\eps}\ll 1$ as for \eqref{eq:energy_approx}), we would have
\[
N \geq \frac{\bar u_0}{c^2 |B| u_\eps(t,x_0, k_0)}.
\]

\subsection{Algorithms}\label{sec:MC_algo}

We discuss in this section  practical aspects of the method. Before stating the algorithm itself, let us emphasize that a key point is to sample efficiently the jumps from $\Pi_\eps$ given by \eqref{def_Pi}. 

Let us fix the current state of the process $D_{h,\eps}$ at a point $z = (x,\hat k)$. In spherical coordinates, $\pi_\eps$ defined in \eqref{def_pi} is equivalent to a probability density function drawing a polar angle $\theta$ and an azimuthal angle $\varphi$. Here, the north pole of the spherical system is the current direction $\hat k$, and it is direct to see that the azimuthal angle $\varphi$ is uniformly distributed over $(0,2\pi)$. We denote this by $\varphi \sim \calU(0,2\pi)$. For the polar angle, a change of variables leads  to $\theta =\arccos(1-\chi)$, where $\chi$ has probability density function 
\[
f_\chi(\chi|x) := \frac{a(\sqrt{2\chi})}{C_\chi \chi^{1+\alpha(x)/2}} \mathbf{1}_{(\eps,2)}(\chi),
\]
and $C_\chi$ is a normalizing constant. Therefore, to draw a jump according to \eqref{def_pi} starting from $\hat k$, we compute
\begin{equation}\label{eq:def_R}
\hat p = R(\theta,\varphi,\hat k):=\cos(\theta)\hat k + \sin(\theta)\Big(I_3 + \sin(\varphi)Q(\hat k) + (1-\cos^2(\varphi))Q^2(\hat k)\Big)\hat k^\perp,
\end{equation}
where $\hat k^\perp$ is an orthonormal vector to $\hat k$, $I_3$ is the $3\times 3$ identity matrix, and 
\be \label{def:Q_mat}
Q(k) = \begin{pmatrix} 0 & -  k_{3} &  k_{2} \\
k_{3} & 0 & - k_{1} \\
- k_{2} &  k_{1} & 0 \end{pmatrix} ,\qquad \text{where}\qquad k=(k_{1}, k_{2}, k_{3}).
 \ee
The transformation $R$ corresponds to a rotation from $\hat k$ to $\hat p$ with polar angle $\theta$ with respect to $\hat k$ and azimuthal angle $\varphi$ with respect to $\hat k^\perp$. Note that the choice of $\hat k^\perp$ is not important since $\varphi$ is uniformly distributed over $(0,2\pi)$. 
 
We notice that in the case of a constant function $a\equiv a_0$, one obtains a truncated Pareto distribution for $\chi$. The corresponding cumulative distribution function can be exactly inverted giving then a direct simulation method. In this case, the cumulative distribution function is given by, for $\chi\in(\eps,2)$,
\[F_\chi(\chi|x) = \frac{a_0}{C_\chi} \int_{\eps}^{\chi} \frac{dv}{v^{1+\alpha(x)/2}} = \frac{1 - (\eps /\chi )^{\alpha(x)/2}}{1 - (\eps /2 )^{\alpha(x)/2}}.\]
The random variable $\chi$ can then be generated by 
\[\chi = F^{-1}_\chi(U|x) = \eps(1- (1 - (\eps /2 )^{\alpha(x)/2}) U )^{-2/\alpha(x)},\]
where $U$ is a random variable uniformly distributed over $(0,1)$ ($U\sim \calU(0,1)$). In the case of a non constant function $a$, the main features of the density $f_\chi(\cdot|x)$ are similar to those of the truncated Pareto distribution, and a stochastic collocation method can be considered to simulate $f_\chi(\cdot|x)$. This method is described in Appendix \ref{sec_sto_col} in our context. It is based on the simulation of the above truncated Pareto distribution and proves to be very effective.   

The algorithm used to simulate a trajectory of $D_{h,\eps}$ can be summarized in the following two procedures. The first one corresponds to the simulation of the diffusion process between two (fictitious) jumps, and we use the notation
\[
S(Z, W, h) = \left\{ \begin{array}{l}
 X + h \, K  \\
 K - 2h \, \sigma^2_{\eps}(X) \, K +  \sqrt{2 h} \, \sigma_{\eps}(X) \,   W \times K,
\end{array} 
\right.\]
with $Z=(X,K)$. Below, $N(0,I_3)$ stands for the three dimensional multivariate Normal distribution with identity covariance matrix.  

\begin{algorithm}[H]
\DontPrintSemicolon
\SetKwInOut{Input}{input}
\SetKwInOut{Output}{output}
\SetKwInOut{Init}{initialization}

\Input{current state of the particle $z=(x,\hat k)$, duration of the diffusion $\delta t$}
\Output{state of the particle after the diffusion process}
\Init{$n\leftarrow [\delta t/h]$ \tcp{number of iterations}
$Z\leftarrow z$  \tcp{initialization of the diffusion state}}
\tcp{Main loop of the diffusion}
\For{$j\leftarrow 1 $ \KwTo $n$}{
$W\sim N(0,I_3)$\;
$Z\leftarrow S(Z, W, h)$\;
$K\leftarrow K/|K|$
 }
 \tcp{Add a diffusion step with stepsize $h' \leq h$ to match the duration $\delta t$}
$h'\leftarrow \delta t - n h$\;
$W\sim N(0,I_3)$\;
$Z\leftarrow S(Z, W, h')$\;
$K\leftarrow K/|K|$\;
 \KwRet{$Z$}
\caption{\texttt{Diffusion}}
\end{algorithm}
\vspace{0.5cm}

The second procedure combines the diffusion step with the jump process. Below, we denote by $\calE(\bar \Lambda_\eps)$ the exponential distribution with parameter $\bar \Lambda_\eps$ defined by \eqref{eq:def_bar_Lambda}.  

\begin{algorithm}[H]
\DontPrintSemicolon
\SetKwInOut{Input}{input}
\SetKwInOut{Output}{output}
\SetKwInOut{Init}{initialization}

\Input{Duration $T$ of the particle evolution}
\Output{state of the particle at time $T$}
\Init{$Z \leftarrow (x,\hat k)\sim \mu_0$ \tcp{initialization of the particle state at random}
$t \leftarrow 0$ \tcp{temporary time variable}
$\delta t  \sim \calE(\bar \Lambda_\eps)$ \tcp{first jump time}}
\tcp{main loop for the path evolution}
\While{$t + \delta t < T$}{
$Z\leftarrow \texttt{Diffusion}(Z, \delta t)$ \;
$U\sim \calU(0,1)$ \;
  \If{$U \leq p(Z)$}{
  \tcp{the jump is accepted, $Z$ is transformed}
   $\chi \sim f_\chi(\cdot \, | \, x)$\;
$\theta \leftarrow \arccos(1-\chi)$ \;
$\varphi \sim \calU(0,2\pi)$\;
$\hat p \leftarrow R(\theta,\varphi,\hat k)$\;
$Z \leftarrow (x, \hat p)$ \;
 }
 $t \leftarrow t + \delta t$\;
 $\delta t \sim \calE(\bar \Lambda_\eps)$\;
}
\tcp{remaining diffusion step of duration $T-t$}
$Z\leftarrow \texttt{Diffusion}(Z, T-t)$\;
\KwRet{Z}
\caption{\texttt{TrajectorySimulation}}
\end{algorithm}
\vspace{0.5cm}


\noindent The rest of the paper is dedicated to numerical simulations and the proofs of our main results.
\section{Validation}\label{sec:test_case}

In this section, we first derive a semi-analytical solution to validate our method in the simplest situation where $\alpha$, $a$, $\lambda$ are  constant functions. We then highlight the crucial role of the small jumps correction for computational efficiency.

\subsection{Semi-analytical solution}

We set $\lambda \equiv 1$ and the RTE \eqref{RTE} reads
\begin{equation}\label{RTE_test}
\partial_t u +  \hat k \cdot \nabla_x u  = \calQ u
\end{equation}
with scattering kernel
\[
\calQ f(\hat k) = a\int_{\mathbb{S}^2} \frac{\sigma(d\hat p)}{|\hat p-\hat k|^{2+\alpha}} (f(\hat p)-f(\hat k)),\qquad \hat k \in \mathbb{S}^2.
\]
Using the the Funk-Hekke formula \cite{Samko}, this operator can be diagonalized in $L^2(\mathbb{S}^2)$ equipped with the inner product
\[
\big<f, g\big>_{L^2(\mathbb{S}^2)} = \int_{\mathbb{S}^2} f(\hat p)\overline{g(\hat p)} \sigma(d\hat p) = \int_0^\pi \int_0^{2\pi} f(\theta,\varphi)\overline{g(\theta,\varphi)}\sin(\theta)d\theta d\varphi.
\]
The eigenvalues are given by
\[
\lambda_l = \frac{a \pi \Gamma(-\alpha/2)}{2^\alpha \Gamma(1+\alpha/2)}\left( \frac{\Gamma(l + 1+ \alpha/2 )}{\Gamma(l + 1 - \alpha/2) } -\frac{\Gamma(1+ \alpha/2 )}{\Gamma(1 - \alpha/2) }   \right)\qquad l\in\mathbb{N},
\]
and the eigenvectors are the spherical harmonics
\[Y_{l,m}(\hat k) = Y_{l,m}(\theta,\varphi) := \sqrt{\frac{(2l+1)(l-m)!}{4\pi(l+m)!}} P^m_l(\cos(\theta))e^{im\varphi},\qquad (l,m)\in \mathbb{N}\times\{-l,\dots,l\},\]
where the $P^m_l$ are the associated Legendre polynomials. In order to derive a semi-analytical solution, we Fourier transform \eqref{RTE_test} w.r.t. $x$, and introduce
\[\hat u (t,q,k) = \int_{\mathbb{R}^3} u(t,x,k)e^{-iq\cdot x}dx.\]
Above, $q=\tilde q := (0,0,\xi)$ so that $\tilde u (t,\xi, \hat k) = \hat u (t,\tilde q, \hat k)$
solves 
\begin{equation}\label{RTE_test_fourier}
\partial_t \tilde u + i \, \hat k\cdot \tilde q \, \tilde u  = \calQ \tilde u.
\end{equation}
Writing $\hat k$ in spherical coordinates with $(0,0,1)$ as north-pole, this latter equation reads,
\[
\partial_t \tilde u(t,\xi,\theta,\varphi) = (  \calQ -i \xi \cos(\theta))\tilde u (t,\xi,\theta,\varphi), \qquad (t,\xi,\theta,\varphi)\in (0,\infty)\times \mathbb{R} \times (0, \pi) \times (0,2\pi).
\]
We now decompose $\tilde u$ on the basis of spherical harmonics
\[
\tilde u(t,\xi,\theta,\varphi) = \sum_{l = 0}^{\infty}\sum_{m=-l}^l \hat u_{l,m}(t,\xi)Y_{l,m}(\theta,\varphi),
\]
resulting in
\begin{align}\label{eq_diff_test}
\frac{d}{dt} \hat u_{l,m} = \lambda_l  \hat u_{l,m} -  i \xi (d^+_{l,m}\hat u_{l+1,m} + d^-_{l,m}\hat u_{l-1,m}) \quad \textrm{for} \quad l \geq 1, \qquad\frac{d}{dt}\hat u_{0,0} =   -i \xi d^+_{0,0}\,\hat u_{1,0}\quad \textrm{for}\quad l=0.
\end{align}
Above, we have used the fact that
\[
\big<Y_{m',l'}, \cos(\theta)Y_{m,l}\big>_{L^2(\mathbb{S}^2)} = \left\{
\begin{array}{lcl}
\displaystyle d^+_{l,m} := \sqrt{\frac{(l+m+1)(l-m+1)}{(2l+1)(2l+3)}} & \text{if} & m=m'\text{ and }l'-l=1, \\
\displaystyle d^-_{l,m} := \sqrt{\frac{(l+m)(l-m)}{(2l-1)(2l+1)}} & \text{if} & m=m'\text{ and }l'-l=-1, \\
0 & & \text{otherwise}.
\end{array}
\right.
\]
For computational purposes, we introduce a cutoff in the variable $l$ ($l\in\{0,\dots,L\}$), and consider a truncated version of \eqref{eq_diff_test} as the vector differential equation
\begin{equation}\label{eq_diff_test2}
\frac{d}{dt}  \hat u^L(t,\xi) = (D^L - i \xi \, A^L) \, \hat u^L(t,\xi), \qquad \hat u^L (t,\xi)= \big(\hat u^L_{l^2 + j}(t,\xi)\big)_{l\in\{0,\dots,L\}\, j\in\{0,\dots,2l\}} \in \mathbb{C}^{(L+1)^2},
\end{equation}
where $D^L$ and $A^L$ are two $(L+1)^2\times (L+1)^2$ matrices defined by
\[
\left\{
\begin{array}{l}
D^L_{l^2 + j+1,l^2 + j+1}:= \lambda_l \qquad\text{for}\qquad l\in \{0,\dots,L\},\quad j\in \{0,\dots,2l\},\\
A^L_{l^2 + j+1,(l + 1)^2 + j+2} :=  d^+_{l,j-l} \qquad\text{for}\qquad l\in \{0,\dots,L-1\},\quad j\in \{0,\dots,2l\}\\
A^L_{l^2 + j+2,(l - 1)^2 + j+1} :=  d^-_{l,j-l+1} \qquad\text{for}\qquad l\in \{1,\dots,L\},\quad j\in \{0,\dots,2(l-1)\}.
\end{array}\right.
\]
All other coefficients in both $D^L$ and $A^L$ are set to $0$. Note that the indexing of the matrices starts at $0$ for simplicity. The solution to \eqref{eq_diff_test2} reads
$
\hat u^L (t,\xi) = e^{(D^L - i \xi \, A^L)t} \hat u^L(0,\xi),
$
where the matrix exponential is computed numerically. For our test case, we consider the following initial condition 
\[
u(t=0,x,\hat k)= \frac{1}{\sqrt{2\pi}} e^{-|x|^2/2} \cdot 2\cos^2(\theta/2) = \frac{1}{\sqrt{2\pi}} e^{-|x|^2/2} ( 2\sqrt{\pi} Y_{0,0}(\theta,\varphi) + 2\sqrt{\pi/3}Y_{0,1} (\theta,\varphi)),
\]
so that 
\[
\hat u^L_{l^2 + j}(t=0,\xi) = \left\{
\begin{array}{ccc}
2\sqrt{\pi} e^{-\xi^2/2} & \text{for} & l=j=0,\\
2\sqrt{\pi/3} e^{-\xi^2/2} & \text{for} & l=j=1,\\
0 & &\text{otherwise}.
\end{array}\right.
\]
Finally, an approximation of $\tilde u$, solution to \eqref{RTE_test_fourier}, is given by
\[
\tilde u^L(t=0,\xi,\theta,\varphi) = \sum_{l=0}^L \sum_{j=0}^{2l}  \Big[e^{(D^L - i \xi \, A^L)t} \hat u^L(0,\xi)\Big]_{l^2 + j}Y_{l,j-l}(\theta,\varphi).
\]
For numerical comparisons with our MC method, we introduce a discretization of the unit sphere $\mathbb{S}^2$ via the polar and azimuthal angles $(\theta_m)_{m}$ and $(\varphi_m)_{m}$, with respective stepsize $\Delta\theta$ and $\Delta \varphi$. We then compare
\[
\tilde u^L (t, \xi, \theta_m, \varphi_m')  \simeq \frac{1}{\Delta \theta \Delta\varphi}\int_{\theta_m}^{\theta_{m+1}}\int_{\varphi_{m'}}^{\varphi_{m'+1}}\tilde u^L (t, \xi, \theta, \varphi)\sin(\theta)  d\theta d\varphi
\]
with its MC approximation
\[
\tilde u^L_N (t, \xi, m, m') = \frac{1}{\Delta \theta \Delta\varphi \, N} \sum_{n=1}^N e^{-i\xi X^n_{3, h,\eps}(t)}\mathbf{1}_{\big(\theta^n_{h,\eps}(t) \in(\theta_m, \theta_{m+1}), \,\varphi^n_{h,\eps}(t) \in (\varphi_{m'},\varphi_{m'+1})\big)} 
\]
where $\theta^n_{h,\eps}$ and $\varphi^n_{h,\eps}$ are respectively the polar and azimuthal angles for $\hat K^n_{h,\eps}$, and where $ (D^n_{h,\eps})_n =  (X^n_{h,\eps}, \hat K^n_{h,\eps})_n$ is a sample of $D_{h,\eps}$ introduced in Section \ref{sec:MC_meth}.

In the following numerical illustrations we consider $a = 0.002$ in the RTE, and set $\Delta \theta = \Delta \varphi = 0.05$, $\eps = 0.1$ and $h = 0.5/\bar \Lambda_\eps\simeq 12.6$ for the approximation parameters. Note that these choices for $\eps$ and $h$ are providing us with a good accuracy at a very low computational cost as we will see. Such values may have to be decreased in other setups and when considering different observables. For instance, in Section \ref{sec:var_alp} where  $\alpha$ is varying, smaller values of $\eps$ and $h$ are needed to capture correctly the solution. 

Also, in the context of singular scattering kernels, the classical notion of scattering mean free time is not informative since it is equal to $0$ (see \eqref{eq:mft}). Instead, we define a characteristic time using the inverse of the second eigenvalue of $\calQ$, i.e. the first non zero eigenvalue, and set $t_c = - \frac{1}{\lambda_1}$. 
We refer to Figure \ref{fig:vp} for the evolution of $t_c$ w.r.t. $\alpha$. 

\begin{figure}[h!]
\begin{center}
\includegraphics[scale=0.30]{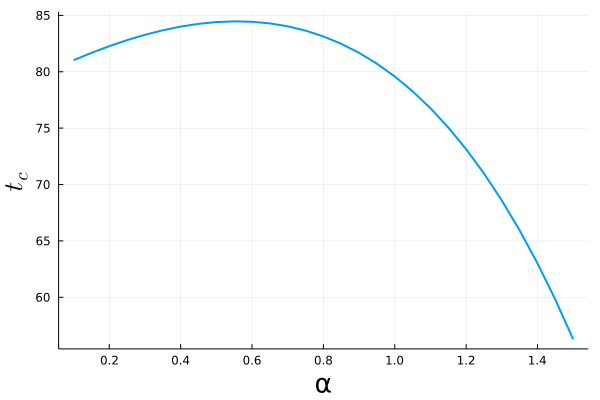} 
\end{center}
\caption{ \label{fig:vp} Illustration of the evolution of $t_c$ w.r.t. $\alpha$.}
\end{figure}
In our setting, $t_c \simeq 79.6$ (for $\alpha=1$), which is about six times the stepsize $h$ needed to capture the diffusive correction. Also, since $\eps$ is not too small, this correction plays a significant role in obtaining the correct dynamics.
\begin{figure}[h!]
\begin{center}
\includegraphics[scale=0.30]{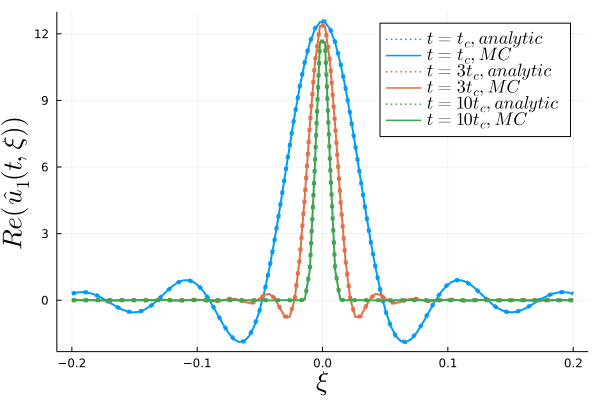} 
\includegraphics[scale=0.30]{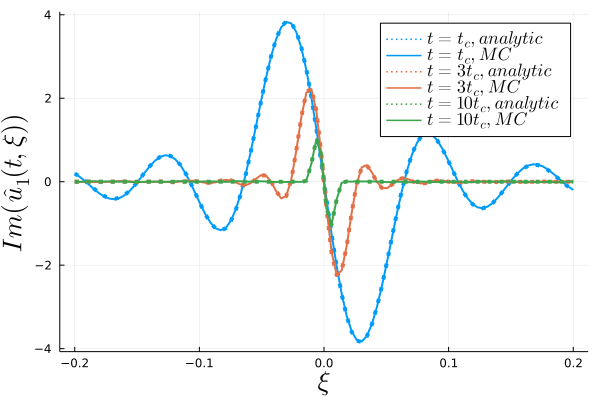} 
\end{center}
\caption{ \label{fig:testcase1} Comparisons of the real and imaginary parts of $\hat u_1(t,\xi)$ for three observation times and for $\alpha=1$. The grid in $\xi$ range from $-0.2$ to $0.2$ with $100$ discretization points and we run $N=2.4 \times 10^6$ particles.}
\end{figure}

In Figure \ref{fig:testcase1}, we compare, for $\alpha=1$, the real and imaginary parts of the observable 
\[\hat u_1(t,\xi) := \int_{0}^\pi \int_{0}^{2\pi} \tilde u^L(t,\xi,\theta,\varphi)\sin(\theta) d\theta d\varphi \qquad \textrm{with} \qquad 
4\pi \, \Delta\theta \Delta \varphi \sum_{m,m'} \tilde u^L_N (t, \xi, m, m'),\]
for three values of $t$. In Figure \ref{fig:testcase2}, we compare the real and imaginary parts of
\[\hat u_2(t,\xi,\theta) := \int_{0}^{2\pi} \tilde u^L(t,\xi,\theta,\varphi) d\varphi \sin(\theta) \qquad \textrm{with} \qquad 4\pi \, \Delta \varphi\sum_{m'} \tilde u^L_N (t, \xi, m, m'),\]
for three values of $\xi$.

\begin{figure}[h!]
\begin{center}
\includegraphics[scale=0.3]{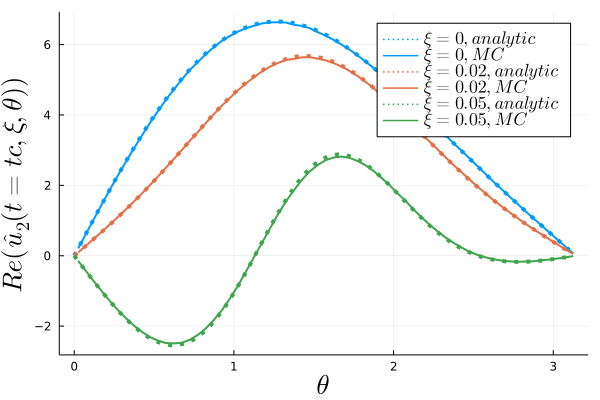} 
\includegraphics[scale=0.3]{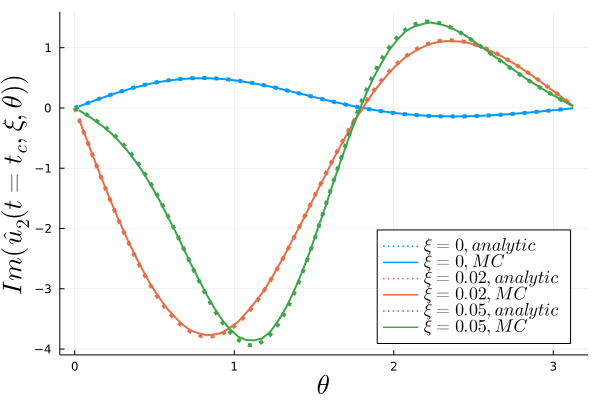} 
\\
\includegraphics[scale=0.3]{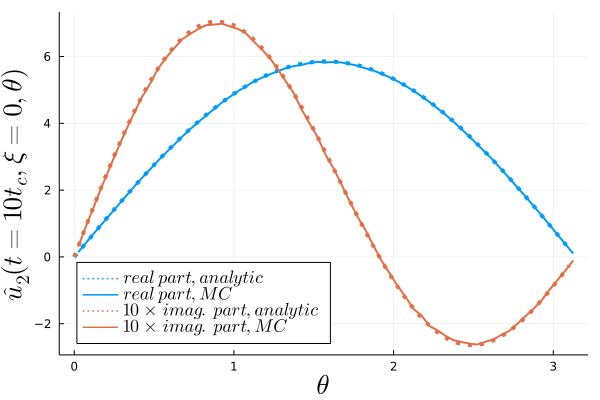} 
\end{center}
\caption{\label{fig:testcase2} Comparisons of the real and imaginary parts of $\hat u_2(t,\xi,\theta)$ for $\alpha=1$, for three values of $\xi$ if $t=t_c$, and for $\xi=0$ if $t=10t_c$. We run $N=2.4 \times 10^6$ particles for the top two pictures and $N=24 \times 10^6$ for the third one.}
\end{figure}

In Figure \ref{fig:testcase3}, we compare
\[\hat u_3(t=2t_c,\xi=0.02,\theta) := \int_{0}^{2\pi} \tilde u^L(t=2tc,\xi=0.02,\theta,\varphi) d\varphi \sin(\theta) \qquad \textrm{with} \qquad 4\pi \, \Delta \varphi\sum_{m'} \tilde u^L_N (t=2t_c, \xi=0.02, m'),\]
for three values of $\alpha$.

\begin{figure}[h!]
\begin{center}
\includegraphics[scale=0.3]{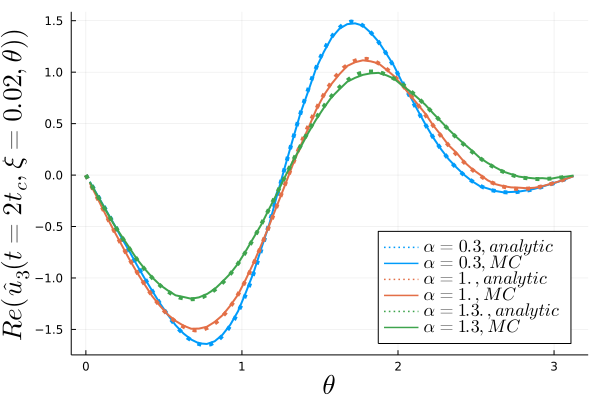} 
\includegraphics[scale=0.3]{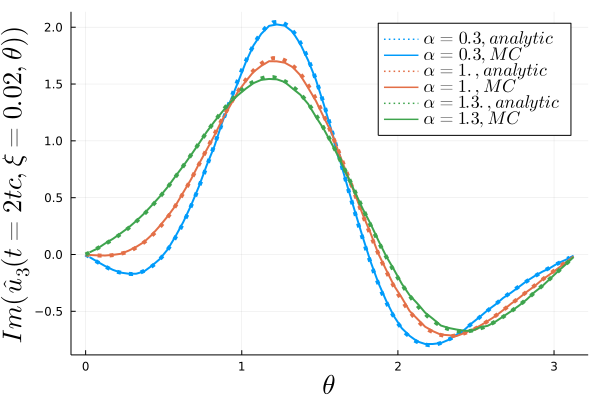} 
\end{center}
\caption{\label{fig:testcase3} Comparisons of the real and imaginary parts of the observable $\hat u_3(t=2t_c,\xi=0.02,\theta)$ for three values of $\alpha$. We run $N=24 \times 10^6$ particles.}
\end{figure}

In all these illustrations, and despite somewhat fairly large values for $\eps$ and $h$, we observe a very good agreement between the Monte Carlo results and the semi-analytic calculations.

\subsection{Role of the correction}

In this section, we highlight the role of the correction provided by the diffusion over the unit sphere w.r.t. the $\hat k$-variable. To this end, we compare the following observables obtained from the semi-analytic solution
\[
u_4(x_3) = \int_0^{3t_c} \tilde u^L (t,x_\perp,x_3,\theta,\varphi) \sin(\theta) \, dt \, dx_\perp \, d\theta \, d\varphi, \qquad x:=(x_\perp,x_3)\in\mathbb{R}^2\times \mathbb{R},
\]
with the ones obtained with our MC method, with and without this diffusive correction, and for various values of $\alpha$, $\varepsilon$ and $h$. The grid in $z$ range from $-300$ to $300$ with size $2^8$ and we run $N = 300 \times 10^6$ particles. According to \eqref{eq:RMSE} and \eqref{eq:rel_err}, the number of samples $N$ is taken large enough so that the RMSE \eqref{def:RMSE} of the MC estimation is of order $10^{-5}$ and the relative MC error \eqref{def:rel_err} is of order $0.03\%$ where $u_4$ takes values of order as low as $10^{-3}$. With this choice of $N$, we can focus our attention on the role played by $\eps$ and $h$ in the approximation.
\begin{figure}[h!]
\begin{center}
\includegraphics[scale=0.30]{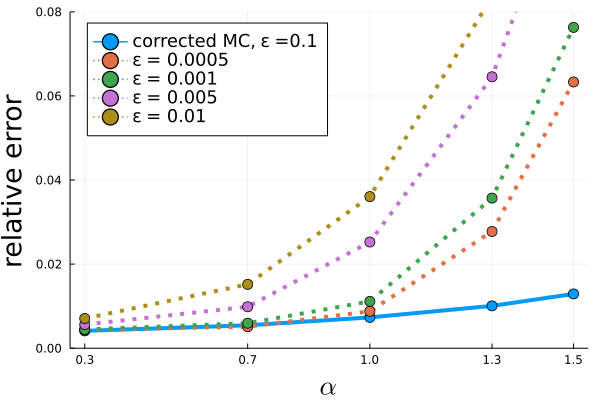} 
\includegraphics[scale=0.30]{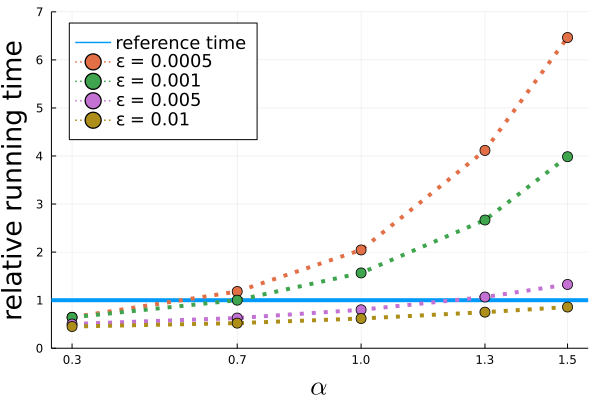} 
\end{center}
\caption{\label{fig:testcase4} Illustration of the relative error $Err_{\eps,\alpha}$ and running time of the MC method with and without a diffusive correction. The reference time in the right picture is the one of corrected method with $\eps=0.1$. }
\end{figure}
In Figure \ref{fig:testcase4}, we represent the relative error
\[
Err_{\eps,\alpha} := \max_z \frac{|u_4(z) - u_{4,MC}(z)|}{\max_z u_4(z)},
\]
for various sizes of the cutoff $\eps$, and where $u_{4,MC}(z)$ is the MC approximation to $u_4(z)$. The left picture illustrates the evolution of the relative error for various $\eps$. The blue curve corresponds to the corrected MC with $\eps = 0.1$ (with still a fairly large stepsize $h=0.5/\bar \Lambda_\eps$) providing at most a relative error slightly larger than $1\%$. The other curves correspond to the noncorrected MC method for several values of $\eps$. The corrected MC consistently yields a better accuracy than the noncorrected version, and even in weakly singular cases where $\alpha$ is less than one, a very small value of $\eps$ (red and green curves) is necessary to match the accuracy of the corrected method. The right picture illustrates the evolution of the relative running time of the noncorrected method w.r.t. the corrected one. For values of $\alpha$ less than 0.7 (weakly singular kernels), corrected and noncorrected methods have similar computational times for comparable accuracy, while in the case of singular kernels with $\alpha \geq 1$, the noncorrected methods yield a much larger cost and a much lower accurary. 

\begin{figure}[h!]
\begin{center}
\includegraphics[scale=0.30]{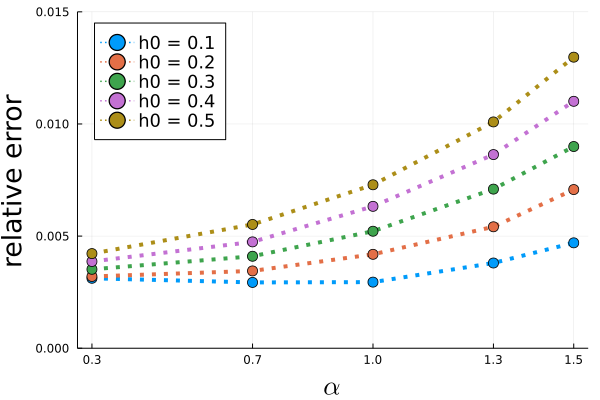} 
\includegraphics[scale=0.30]{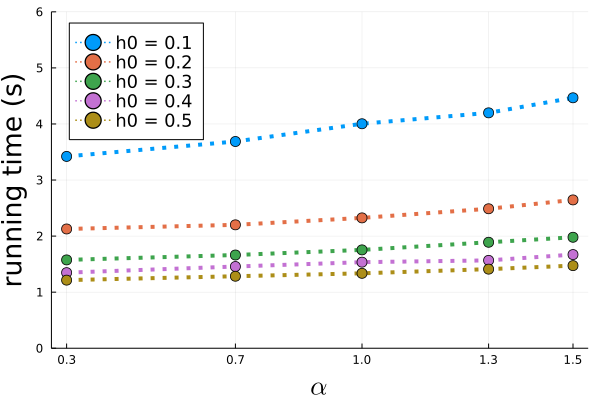} 
\end{center}
\caption{\label{fig:testcase5} Illustration of the relative error $Err_{\eps,\alpha}$ and running time of the (corrected) Monte Carlo method.}
\end{figure}

In Figure \ref{fig:testcase5}, we illustrate the precision and running time sensitivity of the (corrected) MC method w.r.t. the stepsize $h = h_0 /\bar \Lambda_\eps$. As expected, we obtain a better precision for smaller stepsizes but at the price of a longer running time. These effects are amplified as $\alpha$ increases due to the increasing strength of the diffusion correction. In what follows, we select $h_0=0.3$ since this yields a relative error less than $1\%$ for a wide range of $\alpha$'s while not changing significantly the running time.

\begin{figure}[h!]
\begin{center}
\includegraphics[scale=0.30]{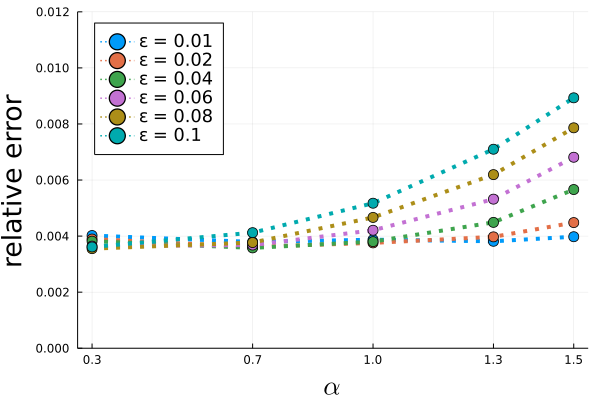} 
\includegraphics[scale=0.30]{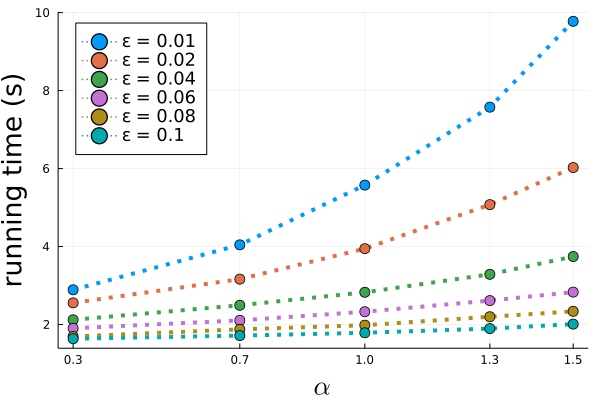} 
\end{center}
\caption{\label{fig:testcase6} Illustration of the relative error $Err_{\eps,\alpha}$ and running time of the (corrected) MC method.}
\end{figure}

In Figure \ref{fig:testcase6}, we depict the precision and running time sensitivity w.r.t. the cutoff parameter $\eps$, and observe the same phenomena as in the case of the stepsize $h$. The parameter $\eps$ defines not only the accuracy of the diffusion correction, but also the average number of jumps, and as a consequence the running time increases as $\eps$ decreases as in the case of the noncorrected Monte Carlo method.


\section{Numerical illustrations}\label{sec:numilla}

\subsection{The role of $\alpha$}\label{subsec:role_alpha}

In this section, we highlight the effects of the kernel singularity on the energy density. We consider  a constant $\alpha$, with $a=0.002$ in this section. Our setting is depicted in Figure \ref{fig:setting_sect_7}. The spatial variable $x$ is decomposed into a main propagation axis $x_3$ and a transverse plane  $x_\perp$, i.e. $x = (x_\perp, x_3)\in \mathbb{R}^2 \times \mathbb{R}$. The same notation holds for the direction variable $\hat k = (\hat k_\perp, \hat k_3)\in\mathbb{S}^2$.
\begin{figure}[h!]
\begin{center}
\includegraphics[scale=0.20]{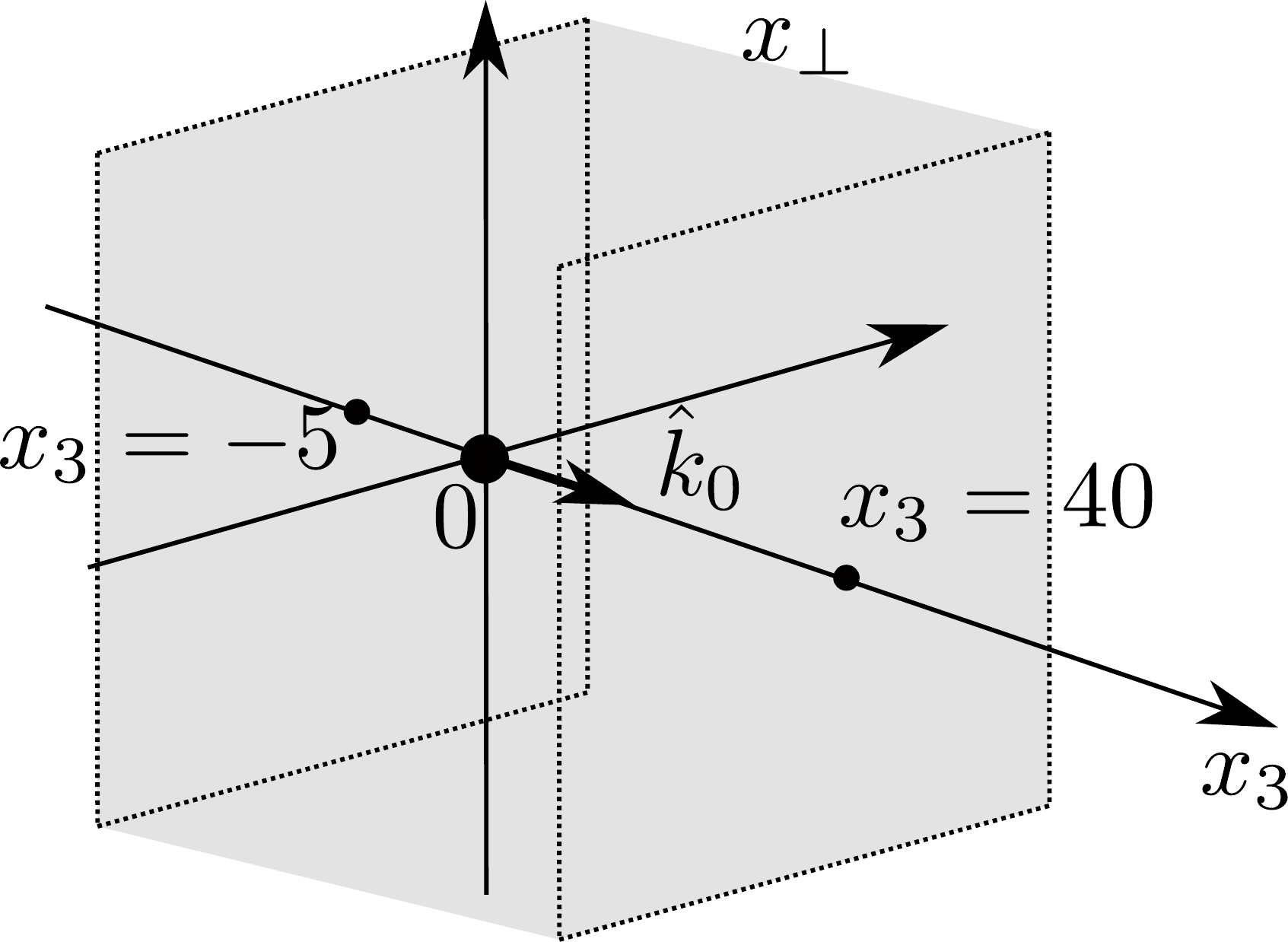} 
\end{center}
\caption{\label{fig:setting_sect_7} Illustration of the numerical setting.}
\end{figure}
We choose an initial condition for \eqref{RTE} of the form
\[
u_0(x,\hat k) = \delta(x) \delta(\hat k - \hat k _0), \qquad 
\hat k _0 = (0, 0, 1),
\]
modeling a source located at $x=0$ and embedded in the random medium, and emitting in the forward $x_3$-direction. We set a function $\lambda$ of the form $\lambda(x) = \mathbf{1}_{(-5,40)}(x_3)$, that defines a scattering layer between $x_3=-5$ and $x_3=40$. In such a configuration, both transmitted and reflected quantities at $x_3=40$ and $x_3 = -5$ are of interest. With our particular choice for $\hat k_0$, what is obtained at $x_3=-5$ is purely due to backscattering.

In the following two subsections, the MC estimations are obtained using $N=1\times 10^9$ particles and a diffusion stepsize $h=0.3/\bar \Lambda_\eps$. We set $\eps=0.01$ for the calculation of transmitted quantities, and $\eps=0.1$ for the reflected ones. For any value of $\alpha$, the observation time we consider is $T=4t_c$, for $t_c$ the critical time computed for $\alpha=1$. 

In the transmission case and when $\eps$ is too large, the mean free time is large as well and it is possible that particles escape the slab without undergoing any jumps, leading to inaccurate results. Hence the choice $\eps=0.01$. A larger value of $\eps$ is acceptable in the calculation of the reflected quantities since the particles exiting early would not have traveled to the plane located at $x_3=-5$, and the error is reduced compared to the transmission case. 

The running times for the time-integrated transmitted ($\eps=0.01$) and reflected ($\eps=0.1$) signals for different values of $\alpha$ are the following:

\begin{center}
\begin{tabular}{c|c|c|c|c|c}
running time (s)  & $\alpha=0.3$ & $\alpha=0.7$ & $\alpha=0.1$ & $\alpha=1.3$ & $\alpha=1.5$ \\
\hline
 $\varepsilon=0.01$  & 9.38 & 15.21 & 24.54 & 42.74 & 65.39 \\
 \hline
 $\varepsilon=0.1$  & 3.61 & 3.72 & 4.12 & 4.92 & 6.01
\end{tabular}
\end{center}

All these running time measurements account also for the transfer of the resulting arrays from the device to the host. We clearly observe a significantly larger running time for smaller values of $\eps$ and large values of $\alpha$. This is due to the increase in scattering events as the mean free time decreases. These computational times correspond to the cost for the MC method to reach the expected accuracy for fixed $\varepsilon$'s and $\alpha$'s. With our choice of $N=10^9$, the RMSEs \eqref{def:RMSE} are of order $10^{-4}$ (resp. $10^{-5}$) for the transmitted (resp. reflected) observables, and the relative errors are of order $1\%$ (resp. $0.1\%$) for the transmitted (resp. reflected) observables taking values of order $10^{-3}$ (resp. $10^{-4}$ upto $10^{-5}$).

\subsubsection{Energy at the boundaries of the transverse plane}

In what follows, the (time-integrated) transverse reflected and transmitted energy are defined by
\[
F^T_{tr}(x_\perp) := \int_0^T dt \int_{\mathbb{S}^2}\sigma(d\hat k)\,  u(t, x_\perp, x_3 = 40, \hat k)  \qquad\text{and}\qquad F^T_{ref}(x_\perp) := \int_0^T dt \int_{\mathbb{S}^2}\sigma(d\hat k)\, u(t, x_\perp, x_3 = -5, \hat k).
\]
The MC estimators for these quantities are given respectively by
\[\hat F^T_{tr}(m,n) := \frac{1}{\Delta x_\perp \,N}\sum_{j=1}^N \mathbf{1}_{\big(X^{j,\perp}_{h,\eps}(\tau^j) \in \square_{mn}, \, X^{j}_{3, h,\eps}(\tau^j) > 40 \big)}, \; \hat F^T_{ref}(m,n) := \frac{1}{\Delta x_\perp \, N}\sum_{j=1}^N \mathbf{1}_{\big(X^{j,\perp}_{h,\eps}(\tau^j) \in \square_{mn}, \, X^{j}_{3, h,\eps}(\tau^j) < -5 \big)}\]
where 
\[
\tau^j:=\inf(t\in [0,T]:\, X^{j}_{3, h,\eps}(t) > 40 \quad\text{ or }\quad X^{j}_{3, h,\eps}(t)<-5),
\]
is the first time the $j$-th particle exits the slab. Note that once a particle escapes, it cannot reenter it since it propagates freely. Above, $(\square_{mn} )_{m,n}$ is a uniform square grid of the traverse plane to the $x_3$-axis. All squares in the grid have area $\Delta x_\perp$. Note that the grid can be different for the transmitted and reflected signals. We have considered for the transverse variable of the transmitted energy a uniform grid over a detector of size $[-10,10]\times [-10,10]$ centered around the $x_3$-axis, and over a detector of size $[-50,50]\times [-50,50]$ for the reflected energy. For both cases, we chose $128\times 128$ grid points. The principle of these estimators is simply to count the number of particles that exit the slab before time $T$ and to record their position in the transverse plane.

In Figure \ref{fig:flux}, we illustrate the transmitted and reflected energy flux, for several values of $\alpha$. We represent the variations w.r.t. the first coordinate of $x_\perp = (x_1,x_2)$, and for two values of $x_2$.

\begin{figure}[h!]
\begin{center}
\includegraphics[scale=0.35]{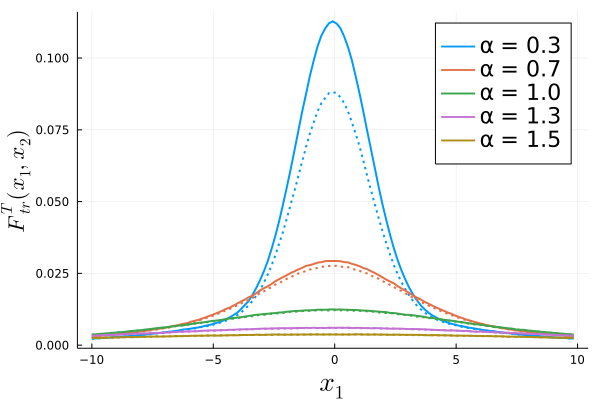}
\includegraphics[scale=0.35]{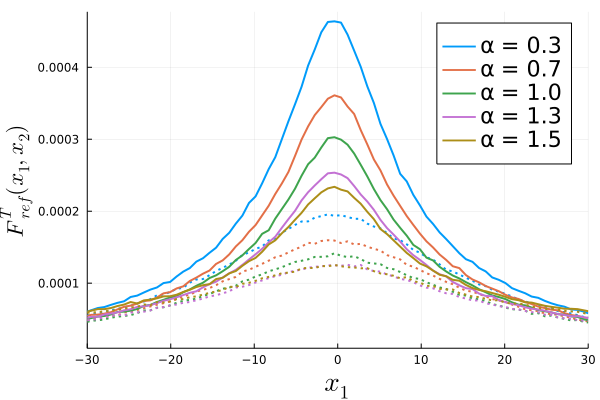} 
\end{center}
\caption{\label{fig:flux} Illustration of the energy at the boundaries $F^T_{tr}(x_\perp)$ and $F^T_{ref}(x_\perp)$ w.r.t. $x_1$ with $x_2=0$ (solid line) and $x_2 = 1$ (dotted line) for the left picture and $x_2 = 10$ (dotted line) for the right one. The RMSEs \eqref{def:RMSE} are less than $6.5\times 10^{-4}$ (resp. $2.6\times 10^{-5}$) on the left picture (resp. right picture), while the relative errors \eqref{def:rel_err} are less than $0.6\%$ for the left picture (resp. $1.3\%$ for the right picture) for values of the observables as low as $10^{-3}$ (resp. $10^{-5}$).}
\end{figure}

One can observe that at fixed times, the larger the $\alpha$, the more diffuse are the signals. Indeed, as $\alpha$ increases, the jump intensity $\Lambda_\eps$ (in other words the number of scattering events) increases as well as the strength of the diffusive correction $\sigma_\eps$ (see Figure \ref{fig:evol}).

\begin{figure}[h!]
\begin{center}
\includegraphics[scale=0.30]{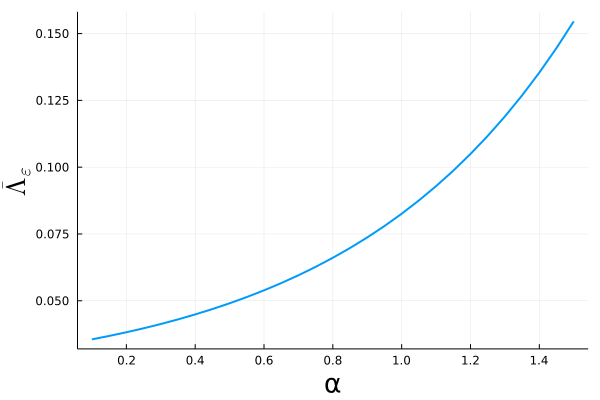} 
\includegraphics[scale=0.30]{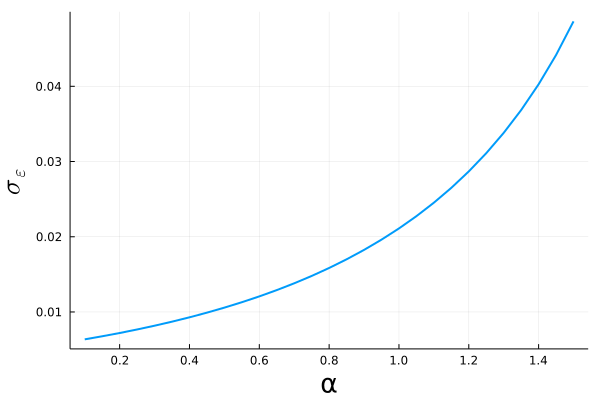} 
\end{center}
\caption{\label{fig:evol} Illustrations of the evolution w.r.t. $\alpha$ of $\Lambda_\eps$ defined by \eqref{eq:Lambda} and $\sigma_\eps$ defined by \eqref{def_sig}. Here, $a=0.002$ and $\eps = 0.01$.}
\end{figure}

\subsubsection{Time evolution of the exiting energy}

Here, we are interested of the time evolution of the energy exiting the slab, and we define the (integrated) reflected and transmitted energy by
\[
\calF_{tr}(t) := \int_{\R^2} dx_\perp \int_{\mathbb{S}^2}\sigma(d\hat k)\,  u(t, x_\perp, x_3 = 40, \hat k)  \qquad\text{and}\qquad \calF_{ref}(t) := \int_{\R^2} dx_\perp \int_{\mathbb{S}^2}\sigma(d\hat k)\,  u(t, x_\perp, x_3 = -5, \hat k).
\]
The MC estimators for these two quantities are given by
\[\hat \calF_{tr}(n) := \frac{1}{dt \, N}\sum_{j=1}^N \mathbf{1}_{\big( \tau^j \in (t_n, t_{n+1}], \, X^{j}_{3, h,\eps}(\tau^j) > 40  \big)}\qquad\text{and} \qquad \hat \calF_{ref}(n) := \frac{1}{dt \, N}\sum_{j=1}^N \mathbf{1}_{\big(\tau^j \in (t_n, t_{n+1}], \, X^{j}_{3, h,\eps}(\tau^j) < -5 \big)}.\]
Here, $(t_{n})_{n}$ is a uniform grid of the time interval with stepsize $dt$. For the transmitted signal, we have considered the time interval $[40, 45]$ with a stepsize $dt=0.02$, and have set $[0,4t_c]$ with a stepsize $dt=0.4$ for the backscattered signal. Note that the time interval starts at $40$ for the transmitted energy, which is the travel time of the wave  (traveling at speed $c_0=1$) from the source to the plane $x_3=40$. These estimators count the number of particles that exit the slab in the time interval $(t_n, t_{n+1}]$ at each side of the slab. In Figure \ref{fig:time_flux}, we illustrate the evolution of the transmitted and reflected energy, for several values of $\alpha$. 

\begin{figure}[h!]
\begin{center}
\includegraphics[scale=0.30]{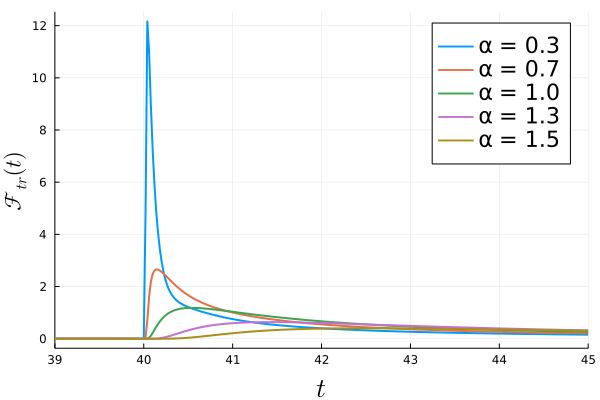} 
\includegraphics[scale=0.30]{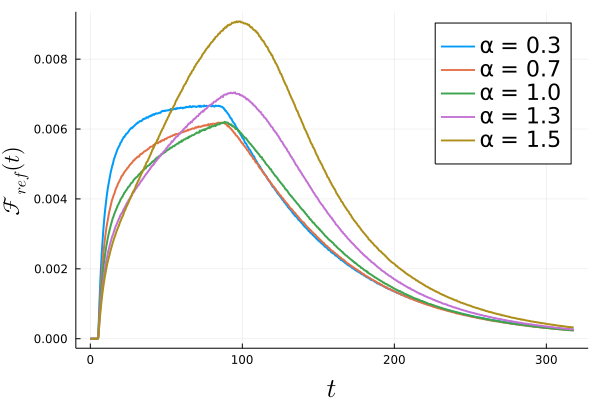} 
\end{center}
\caption{\label{fig:time_flux} Time evolution of the energy at the boundaries $\calF_{tr}(t)$ and $\calF_{ref}(t)$. The RMSE \eqref{def:RMSE} are less than $8\times 10^{-4}$ (resp. $4\times 10^{-5}$) on the left picture (resp. right picture), while the relative errors \eqref{def:rel_err} are less than $0.7\%$ for the left picture (resp. $0.5\%$ for the right picture) for values of the observables as low as $10^{-3}$ (resp. $10^{-4}$).}
\end{figure}

In the case of the transmitted signal (left), and for small values of $\alpha$, we see the arrival of the coherent wave at the proper travel time followed by the coda. When $\alpha$ increases, one notices the stronger impact of scattering and of the diffusive correction that smooths the signal out and damps its amplitude. For the largest $\alpha$, we only observe a coda. Regarding the reflected signal (right), there is only a coda for all $\alpha$ due our choice of $\hat k_0$, and one can observe two stages in the dynamics: backscattering increases up to a time of order $t_c$, about which exponentially decay due to the operator $\calQ$ takes over. 


\subsubsection{Comparison with the Henyey-Greenstein scattering kernel}

In this section,  we compare the solutions to the RTE with Henyey-Greenstein scattering kernel \fref{HG} for an anisotropy factor $g$ close to one with the solutions to \eqref{RTE_approx} with singular kernel derived from \eqref{eq:lim_HG}, that is by setting $a \equiv (1-g)/(2\pi)$ and $\alpha=1$ in \fref{def:rho}. 
Note that the value of the constant $a$ changes with $g$, and as a consequence $\bar \Lambda_\eps$, $h$, and $\sigma_\eps$ vary accordingly. To illustrate this approximation, we still consider the setting depicted in Figure \ref{fig:setting_sect_7} and the various observables introduced in the previous sections, but now at a time $T=300$.

\begin{figure}[h!]
\begin{center}
\includegraphics[scale=0.30]{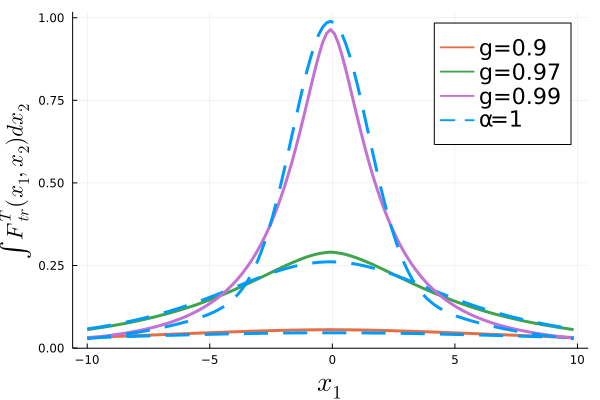} 
\includegraphics[scale=0.30]{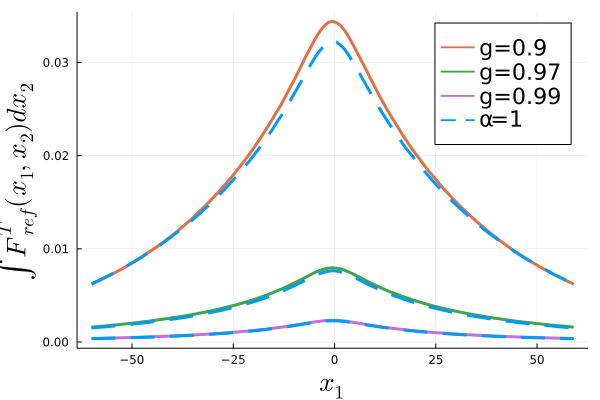} \\
\includegraphics[scale=0.30]{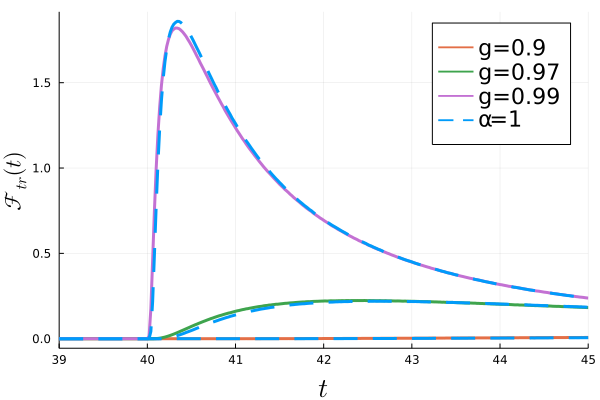} 
\includegraphics[scale=0.30]{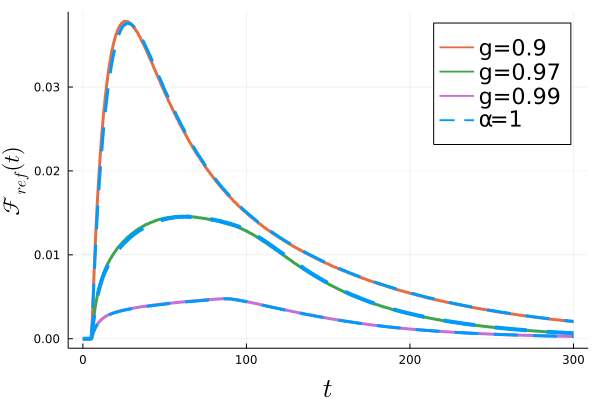}
\end{center}
\caption{\label{fig:HG} Comparison of the observables obtained using the Henyey-Greenstein scattering kernel and our singular kernel with $\alpha=1$, $T=300$, $\eps=0.01$ for the transmitted observables (left panels), and $\eps=0.1$ for the reflected ones (right panels). The RMSEs \eqref{def:RMSE} and relative errors \eqref{def:rel_err} are similar to those of Figures \ref{fig:flux} and \ref{fig:time_flux}.}
\end{figure}

We observe in Figure \ref{fig:HG} the very good agreement between the two solutions. The reflected signal is well captured by our method despite fairly large values of $\eps$ and $h$.  Also, let us mention that the computational cost is decreasing as the anisotropic parameter $g$ is getting close to $1$, as the overall jump intensity decreases in this case in the highly peaked regime $g \to 1$. Regarding the transmitted signal, $\eps$ (and then $h$) needs to be lowered for an accurate approximation, as explained at the beginning of Section \ref{subsec:role_alpha}. 

The RTE with a Henyey-Greenstein scattering kernel is simulated with a standard MC method. Compared to our method, its computational costs to achieve RMSEs of order $10^{-4}$ and $10^{-5}$ for respectively the transmitted and reflected observables are the following:
\begin{center}
\begin{tabular}{c|c|c|c}
running time (s)   & $g=0.97$ & $g=0.98$ & $g=0.99$  \\
\hline
HG kernel  & 8.3 & 7.6 & 6.0 \\
\hline
singular kernel, $\eps = 0.01$ & 13.9 & 6.7 & 2.0\\
\hline
singular kernel, $\eps = 0.1$ & 2.5 & 1.5 & 0.7
\end{tabular}
\end{center}
Here, $\eps=0.01$ is considered for the transmitted observables, while we set $\eps=0.1$ for the reflected ones. According to this table, lower computational times are observed with our method for the three considered $g$'s compared to standard MC methods for the Henyey-Greenstein scattering kernel. Our MC method provides therefore an efficient tool to simulate an RTE with a Henyey-Greenstein kernel. For transmitted observables, $g$ needs to be quite close to one to provide a significant advantage to our method.

\subsection{Varying $\alpha$ function}\label{sec:var_alp}

In this section, we investigate the influence of a varying $\alpha$ function that characterizes the strength of the singularity. We consider two situations, one inspired from optical tomography, and the second one from wave propagation through atmospheric turbulence.

\subsubsection{A two-stage model with a sphere}

We keep the setting introduced in Section \ref{subsec:role_alpha}, and add a defect with a different value of $\alpha$ to the setting. This defect is modeled by ball of radius 3 centered at the origin and where $\alpha$ is equal to $\alpha_1$. We set $\alpha\equiv 1$ in the exterior of the ball, corresponding to the peak forward regime of the Henyey-Greenstein scattering kernel. See Figure \ref{fig:setting_sect_81}. This situation models a biological tissue in which statistical properties are changing and define a region of interest for imaging. 

\begin{figure}[h!]
\begin{center}
\includegraphics[scale=0.20]{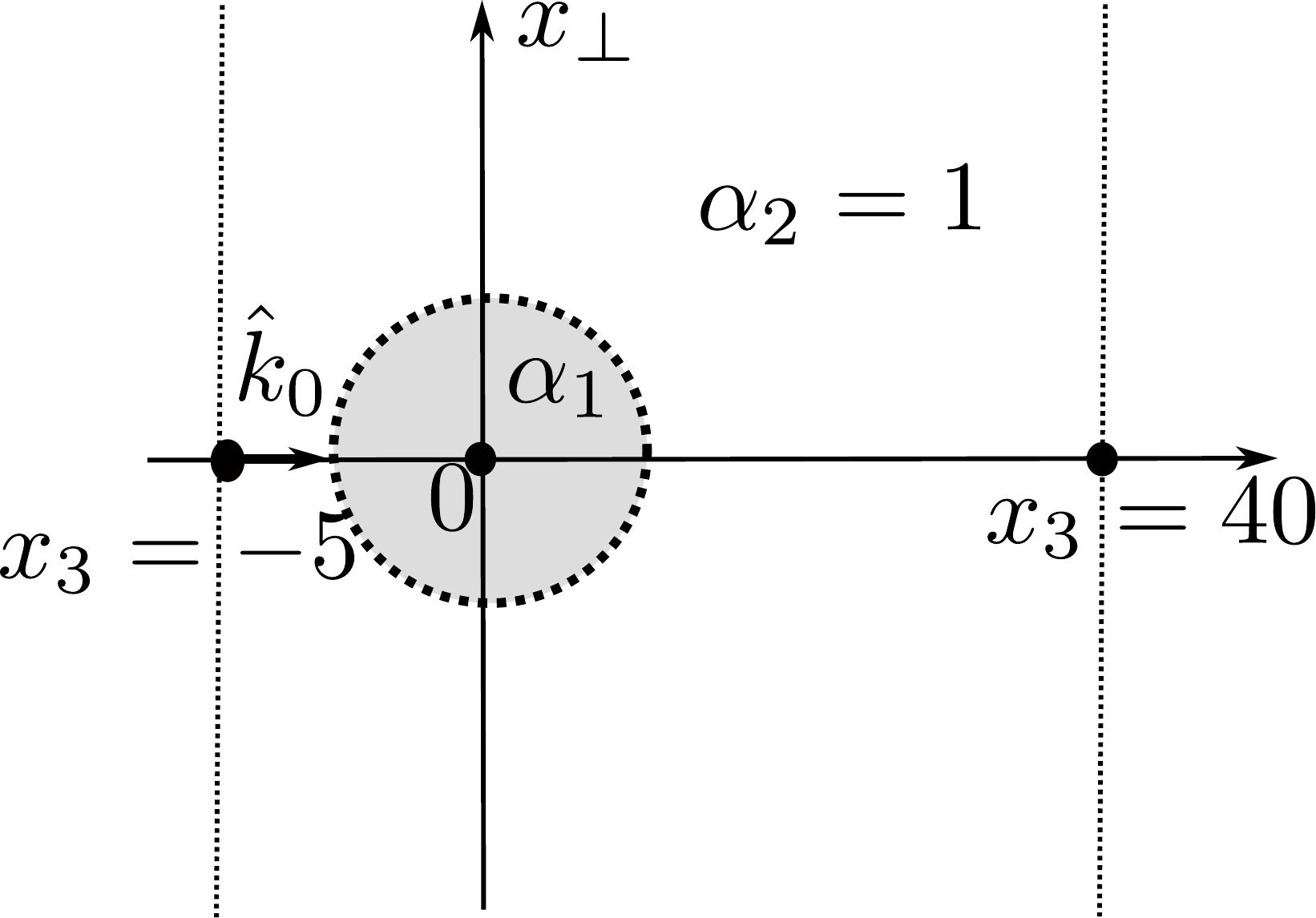}
\end{center}
\caption{\label{fig:setting_sect_81} Illustration of the setting with $\lambda = \mathbf{1}_{\{x_3\in (-5, 40)\}}$ and $\alpha(x) = \alpha_1 \mathbf{1}_{x\in B} + 1\cdot  \mathbf{1}_{x\not\in B}$ where $B$ is a ball centered at $0$ with radius 3.}
\end{figure}

We illustrate in Figure \ref{fig:tomo} the impact of the introduction of the defect on the observables introduced in Section \ref{subsec:role_alpha}. The impact is stronger on transmitted observables and quite significant, giving then the possibility to identify the defect with $\alpha=\alpha_1$ inside the scattering medium. Reflected quantities tend to be less sensitive to the presence of the defect since a fraction of the signal is backscattered before reaching it.  

  
\begin{figure}[h!]
\begin{center}
\includegraphics[scale=0.3]{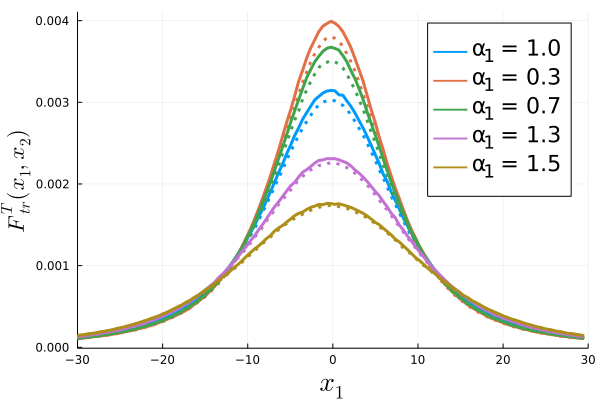} 
\includegraphics[scale=0.3]{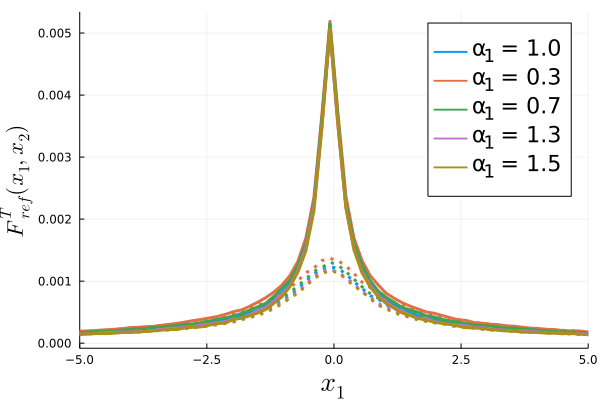} \\
\includegraphics[scale=0.3]{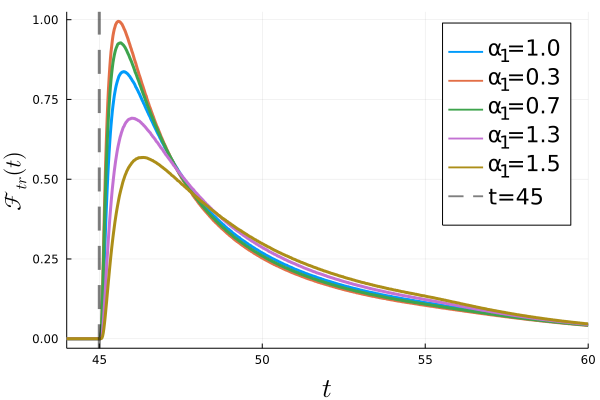} 
\includegraphics[scale=0.3]{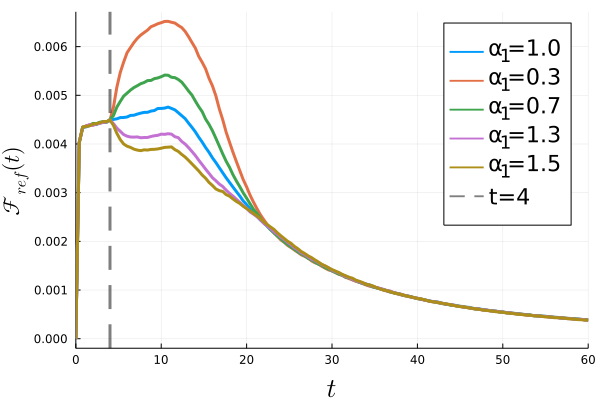}
\end{center}
\caption{\label{fig:tomo} Illustrations the transmitted (left-hand-side) and reflected (right-hand-side) observables with $T=300$, and $\eps=0.01$. For the top two pictures we set $x_2=0$ (solid lines) and $x_2=1.5$ for the top-left and $x_2=0.5$ for top-right picture (dotted lines).}
\end{figure}

\subsubsection{Non-Kolmogorov turbulences}\label{sec:NKt}

In this section, we keep once more the setting introduced in Section \ref{subsec:role_alpha}, with the difference that $\alpha$ takes three different large values depending on the altitude parametrized by $x_3$, see Figure \ref{fig:turb}:
\[
\alpha(x_3) = 5/3 \cdot \mathbf{1}_{\{x_3\leq 2\}} + 4/3 \cdot \mathbf{1}_{\{2<x_3\leq 8\}}+ 1.9\cdot \mathbf{1}_{\{8<x_3\}}.
\]
The value $5/3$ corresponds to standard Kolmogorov turbulences, while other values are associated with non-Kolmogorov turbulence models \cite{belenkii, stribling, zilberman}. In these models, it is considered that for altitudes higher than 8km, the atmospheric turbulence yields larger statistical patterns (which tend to be created by singular kernels) than around the ground (0-2km). Hence, we set $\alpha=1.9$ for altitudes greater than 8km. The function $a$ is no longer constant in these models, and for our illustrations we chose 
\[a(r) = 0.002\cdot\exp(-r^2/(2\times 0.8^2)).\]

In Figure \ref{fig:turbs}, one can notice that non-Kolmogorov turbulence yields quite different signals compared to Kolmogorov turbulence, in particular for reflected quantities. As we saw in Section \ref{subsec:role_alpha}, the higher the $\alpha$, the more diffuse is the signal which then enhances reflected signals. This explains the increased reflections in the non-Kolmogorov case.



\begin{figure}[h!]
\begin{center}
\includegraphics[scale=0.30]{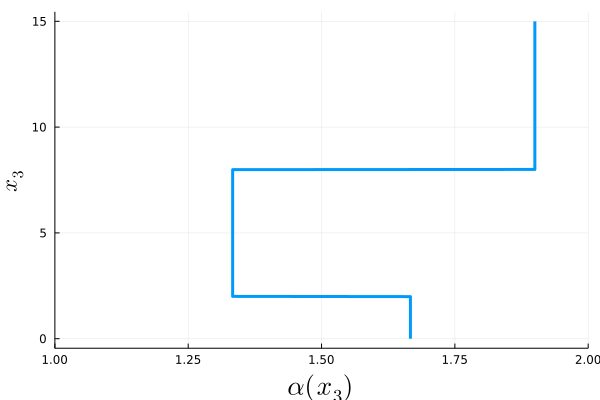} 
\end{center}
\caption{\label{fig:turb} Illustration of a three stages $\alpha$-profile for a non-Kolmogorov phase function.}
\end{figure}

\begin{figure}[h!]
\begin{center}
\includegraphics[scale=0.3]{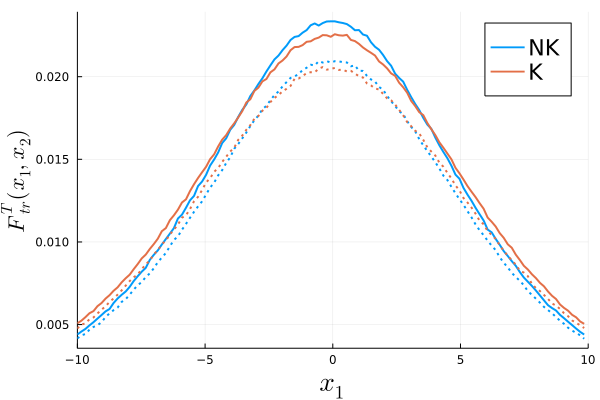} 
\includegraphics[scale=0.3]{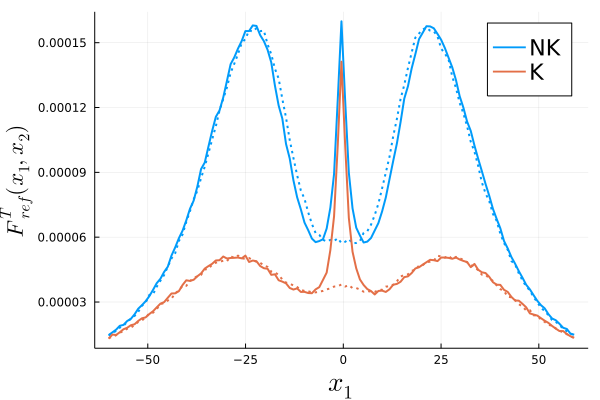} \\
\includegraphics[scale=0.3]{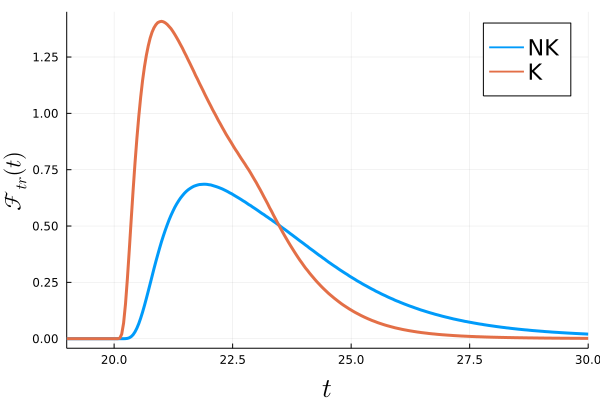} 
\includegraphics[scale=0.3]{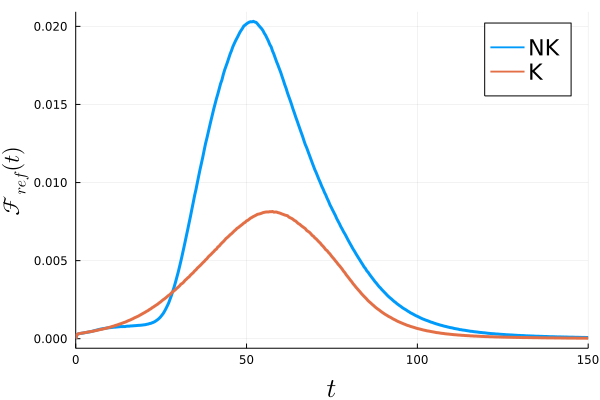}
\end{center}
\caption{\label{fig:turbs} Illustrations the transmitted (left-hand-side) and reflected (right-hand-side) observables with $T=300$, and $\eps=0.01$. For the top two pictures we illustrate $x_2=0$ (solid lines) and $x_2=2$ for the top-left and $x_2=5$ for top-right picture (dot line).}
\end{figure}

\section{Proofs} \label{secproofs}

This section is dedicated to the proof of Proposition \ref{prop:approx_RTE}, describing the approximation of the RTE \eqref{RTE} by \eqref{RTE_approx} where the small jumps have been replaced by a diffusion term, Proposition \ref{rep_prob2}, providing the probabilistic representation to \eqref{RTE_approx}, and Theorem \ref{mainth}, justifying the overall MC method involving a discretization scheme for the diffusion part.

\subsection{Proof of Proposition \ref{prop:approx_RTE}}\label{proof:prop_approx_RTE}

 
Let $v_\eps := u_\eps - u$, so that $v_\eps(t=0)=0$. We have 
\[\begin{split}
    \frac{d}{dt} \|v_\eps(t)\|^2_{L^2(\mathbb{R}^3\times \mathbb{S}^2)} & = 2 <\partial_t v_\eps(t), v_\eps(t)>_{L^2(\mathbb{R}^3\times \mathbb{S}^2)}\\
    &= 2 <(\sigma_{\eps}^2\Delta_{\mathbb{S}^2}+\mathcal{L}^\eps_{>}) v_\eps(t), v_\eps(t)>_{L^2(\mathbb{R}^3\times \mathbb{S}^2)} +  2 < (\sigma_{\eps}^2\Delta_{\mathbb{S}^2} - \mathcal{L}^\eps_{<})u, v_\eps(t)>_{L^2(\mathbb{R}^3\times \mathbb{S}^2)}.
\end{split}\]
Since $\Delta_{\mathbb{S}^2}$ is a nonpositive operator, we have 
\[\begin{split}
<(\sigma^2_{\eps}\Delta_{\mathbb{S}^2}+\mathcal{L}^\eps_{>}) v_\eps(t), v_\eps(t)>_{L^2(\mathbb{R}^3\times \mathbb{S}^2)}  &\leq - \frac{1}{2} \iiint_{\mathbb{R}^3\times \mathbb{S}^2\times \mathbb{S}^2 } dx\, \sigma(d\hat p)\, \sigma(d\hat k) \frac{\lambda(x)}{2^{1+\alpha(x)/2}} \int_{S^\eps_>} \rho(x, \hat p \cdot \hat k)(f(x, \hat p) - f(x, \hat k))^2\\& \leq 0.
\end{split}\]
We then obtain
\[
\sup_{t\in[0,T]}\|v_\eps(t)\|^2_{L^2(\mathbb{R}^3\times \mathbb{S}^2)}  \leq 2 \int_0^T \|(\sigma^2_{\eps}\Delta_{\mathbb{S}^2} - \mathcal{L}^\eps_{<})u(t)\|^2_{L^2(\mathbb{R}^3\times \mathbb{S}^2)} dt,
\]
which concludes the proof using the following lemma.
\begin{lemma}\label{lemma:approx_sig} Let $0<\eps<\eps_0<1$. Then, for any $f\in  L^2_x(\mathbb{R}^3,\mathcal{C}^4_{\hat k}(\mathbb{S}^2))$, we have
\[
\|(\calL^\eps_{<} - \sigma^2_{\eps} \Delta_{\mathbb{S}^2}) f\|_{L^{2}(\mathbb{R}^3 \times \mathbb{S}^2)} \leq \eps^{2-\alpha_M/2} \, E(f)
\]
where, with $\check{f} (x,v) := f(x, v/|v|)$ for $ v \in \mathbb{R}^3$,
\begin{equation}\label{eq:def_C_eps}
\begin{split}
 E(f) & := \frac{\pi}{3(1-\eps_0)^6}  \sup\lambda \sup a \sup_{ |h| \leq r_{\eps_0}}\|D^4_k \check f (\cdot, \cdot +  h)\|_{L^2(\mathbb{R}^3\times \mathbb{S}^2)} \\
 & + \Big( \frac{6 \pi}{(1-\eps_0)^3} \sup a + 2^{4}\pi \sup_{v\in[0,2\sqrt{2\eps_0}]} |a''(v)|\Big)  \sup\lambda \, \|\Delta_{\mathbb{S}^2} f\|_{L^2(\mathbb{R}^3\times \mathbb{S}^2)}.
\end{split}
\end{equation}
\end{lemma}
\begin{proof}
Before starting the proof, we introduce the \emph{retraction} $R_{\hat k}$ at $\hat k$ onto the sphere $R_{\hat k}(h):=\frac{\hat k + h}{|\hat k + h|}$, 
and 
\[B_{\eps,\hat k} := R^{-1}_{\hat k}(S^\eps_{<}) = \big\{ h = \beta_1\,\hat k_1^\perp +\beta_2\,\hat k_2^\perp : \quad \beta = (\beta_1,\beta_1)\in \mathbb{R}^2\quad \text{with}\quad |\beta| < r_\eps\big\},\]
where $(\hat k_1^\perp, \hat k_2^\perp)$ stands for an orthonormal basis of $\hat k^\perp$. We also recall that $r_\eps=\sqrt{1-(1-\eps)^2}/(1-\eps)$,
coming from the relation $\tan(\arccos(s)) = \sqrt{1-s^2}/s$ and \eqref{def:Seps}. In different terms, $B_{\eps,\hat k}$ is a ball centered at $0$ with radius $r_\eps$ on the tangent plane to the unit sphere at $\hat k$, and the retraction $R_{\hat k}$ holds from $B_{\eps,\hat k}$ onto $S^\eps_{<}$. 
 
To prove the lemma, we start with the following change of variables $\hat p = R_{\hat k}(h)$ in $\calL^\eps_{<}$, so that
\[
\calL^\eps_{<} f(x, \hat k) = \frac{\lambda(x)}{2^{1+\alpha(x)/2}}\int_{B_{\eps,\hat k}} \rho(x, R_{\hat k}(h) \cdot \hat k)\big(f(x, R_{\hat k}(h)) - f(x, R_{\hat k}(0))\big) |\det \text{Jac} R_{\hat k}(h)| dh. 
\] 
Using that $ \check f(x, \hat k + h) = f(x, R_{\hat k}(h))$ and $ \check f(x, \hat k) = f(x, \hat k)$, one can decompose $\calL^\eps_{<} f$ as
\[
\calL^\eps_{<} f(x,\hat k) = D_1 + D_2 + D_3 + D_4,
\] 
where the terms $D_j$ follow with obvious notations from the Taylor expansion
\[\begin{split}
\check f(x,\hat k + h) - \check f (x,\hat k) & = D_k\check f(x,\hat k)(h) +\frac{1}{2!} D^2_k\check f(x,\hat k)(h,h) + \frac{1}{3!} D^3_k\check f(x,\hat k)(h,h,h) \\
&+ \frac{1}{3!} \int_0^1 (1-s)^3 D^4_k\check f(x,\hat k + s h)(h,h,h,h) ds.
\end{split}\] 

\paragraph{The terms  $D_1$ and $D_3$.}
Using that the ball $B_{\eps,\hat k}$ in the tangent plane is symmetric with respect to $0$, we just make the change of variables $h\to -h$, so that $D_1 = -D_1$ and $D_3 = -D_3$ leading to $D_1 = D_3 = 0$.

\paragraph{The term $D_4$.} 
We have
\[
|D_4|  \leq \frac{\lambda(x)}{3!\,2^{1+\alpha(x)/2}} \int_0^1 ds\, (1-s)^3 \int_{B_{\eps,\hat k}} dh\,\rho(x, R_{\hat k}(h) \cdot \hat k) \|D^4_k \check f(x,\hat k + s h)\| \, |h|^4 |\det \text{Jac} R_{\hat k}(h)|\,dh.\]
Since $
R_{\hat k}(h) = \frac{\hat k + h}{\sqrt{1 + |h|^2}},
$
we have
\[
R_{\hat k}(h) \cdot \hat k = \frac{1}{\sqrt{1 + |h|^2}} \qquad \text{and} \qquad |\det \text{Jac} R_{\hat k}(h)| = \frac{1}{\sqrt{1+|h|^2}}.
\]
As a consequence, we find 
\[\begin{split}
\|D_4\|_{L^2(\mathbb{R}^3\times \mathbb{S}^2)} & \leq \frac{1}{4!} \sup_{v} a(v) \sup_{|h|\leq r_\eps}\|D^4 \check f (\cdot, \cdot + h)\|_{L^2(\mathbb{R}^3\times \mathbb{S}^2)} \\
&\times \sup_{x} \frac{\lambda(x)}{2^{1+\alpha(x)/2}}  \int_{\{|h|\leq r_\eps\}}  \frac{|h|^4}{ (1- 1/\sqrt{1 + |h|^2})^{1+\alpha(x)/2}} \,dh.
\end{split}\]
Changing to polar coordinates in the last integral gives
\[\begin{split}
\int_{\{|h|\leq r_\eps\}}  \frac{|h|^4}{ (1- 1/\sqrt{1 + |h|^2})^{1+\alpha(x)/2}} \,dh & = 2\pi \int_0^{r_\eps}  \frac{r^5}{ (1- 1/\sqrt{1 + r^2})^{1+\alpha(x)/2}} dr  \\
& = 2\pi\int_0^{\eps/(1-\eps)} v^{1-\alpha(x)/2}(2+v)^2(v+1)^{2+\alpha(x)/2} dv
\end{split}\]
where we used the change of variables $ v=\sqrt{1+r^2}-1$ and that $\sqrt{1+r^2_\eps}-1 = \eps/(1-\eps)$. This gives finally
\[
\|D_4\|_{L^2(\mathbb{R}^3\times \mathbb{S}^2)}  \leq \frac{\pi }{3} \sup_{v} a(v)\sup_{h \in B_\eps}\|D^4 \hat f (\cdot, \cdot + h)\|_{L^2(\mathbb{R}^3\times \mathbb{S}^2)}\sup_{x} \lambda(x)  \frac{\eps^{2-\alpha_M/2}}{(1-\eps)^6}.
\]

\paragraph{The term $D_2$.}

For this last term, we have
  \[
D_2 = \frac{\lambda(x)}{2^{2+\alpha(x)/2}}\int_{B_{\eps,\hat k}} \rho(x, R_{\hat k}(u) \cdot \hat k) D^2\check f(x,\hat k)(h, h) |\det \text{Jac} R_{\hat k}(h)|dh,
\]
with
\[
D^2\check  f(x,\hat k)(h, h) = h_1^2 \, \partial^2_{k_1k_1}\check  f(x,\hat k) + h_2^2 \, \partial^2_{k_2k_2}\check  f(x,\hat k) + 2h_1 h_2 \, \partial^2_{k_1k_2}\check  f(x,\hat k),
\]
and, accordingly, the following decomposition $
D_2 = D_{21} + D_{22} + 2 D_{23}.$ 
Applying the change of variables $h=(h_1,h_2)\to (-h_1,h_2)$ leads to $D_{23}=0$. Setting $h=(h_1,h_2)\to (h_2,h_1)$ leads to 
\[D_2 = \frac{\lambda(x)}{2^{3+\alpha(x)/2}}\int_{B_{\eps,\hat k}}  \rho(x, R_{\hat k}(h) \cdot \hat k) |h|^2  |\det \text{Jac} R_{\hat k}(h)| dh \, Trace(Hess\check f(x,\hat k)),\]
where
\[ Trace(Hess\check f(x,\hat k)) = \Delta_p  f\Big(x,\frac{p}{|p|}\Big)_{| p=\hat k} = \Delta_{\mathbb{S}^2} f(x, \hat k).\]
Furthermore, with the change of variables $\hat p = R_{\hat k}(h)$, we find
\[
\int_{B_{\eps,\hat k}}  \rho(x, R_{\hat k}(h) \cdot \hat k) |h|^2  |\det \text{Jac} R_{\hat k}(h)| dh = \int_{S^\eps_<}  \rho(x, \hat p \cdot \hat k) |R^{-1}_{\hat k}(\hat p)|^2  \sigma(d\hat p),
\]
and note that for $\hat p \cdot \hat k = \cos(\theta)$, we have  $
|R^{-1}_{\hat k}(\hat p)|^2 = \tan^2(\theta).
$
As a result, moving to spherical coordinates, and performing the change of variables $v=\tan(\theta/2)$ together with the relation
\[\arccos(s) = 2\arctan\Big(\frac{\sqrt{1-s^2}}{1+s}\Big)\qquad \text {for}\quad s\in(-1,1],\] 
we find, with $r'_\eps=\sqrt{1-(1-\eps)^2}/(2-\eps)$,
\[\begin{split}
\int_{S^\eps_<}  \rho(x, \hat p \cdot \hat k) |R^{-1}_{\hat k}(\hat p)|^2  \sigma(d\hat p) & = 2\pi \int_0^{\arccos(1-\eps)} \rho(x, \cos(\theta))\tan^2(\theta)\sin(\theta)  d\theta \\
 & =  2^{3-\alpha(x)/2} \pi \int_0^{r'_\eps} \frac{a(2 v/\sqrt{1+v^2})(1+v^2)^{\alpha(x)/2}}{(1-v^2)^2v^{\alpha(x)-1}}dv,
\end{split}\]
leading to
  \[
D_2 =\tilde D_2\Delta_{\mathbb{S}^2} f(\hat k) := 2^{1-\alpha(x)} \pi \lambda(x) \int_0^{r'_\eps} \frac{a(2 v/\sqrt{1+v^2})(1+v^2)^{\alpha(x)/2}}{(1-v^2)^2 v^{\alpha(x)-1}} dv \, \Delta_{\mathbb{S}^2} f(x,\hat k).
\]
Now, let us introduce
\[
\tilde \sigma^2_{\eps}(x) :=2^{1-\alpha(x)} \pi \lambda(x)  \int_0^{r'_\eps} a(2 v/\sqrt{1+v^2})\frac{dv}{v^{\alpha(x)-1}},
\]
and remark that
\[\begin{split}
|\tilde D_2 - \tilde \sigma^2_{\eps}(x) | & \leq 3\cdot 2^{3-\alpha(x)} \pi \lambda(x)\sup_{v} a(v) \frac{1}{(1-{r'_\eps}^2)^3} \int_0^{r'_\eps} v^{3 - \alpha(x)} dv \leq 3\cdot 2^{2-\alpha(x)} \pi \lambda(x) \sup_{v} a(v) \frac{{r'_\eps}^2}{(1-{r'_\eps}^2)^3} \\
&\leq 6 \pi \sup_{x}\lambda(x) \sup_{v} a(v) \frac{\eps^{2-\alpha_M/2}}{(1-\eps)^3},
\end{split}\]
since $r'_\eps \leq \sqrt{2\eps}$, $0 \leq \alpha_m\leq \alpha(x)\leq \alpha_M<2$ and $1 - {r'_\eps}^2 = \frac{2(1-\eps)}{2-\eps}>2(1-\eps).$
With the definition of $\sigma_\eps$ given in \fref{def_sig}, 
we obtain using $a'(0)=0$,
\[\begin{split}
| \tilde \sigma^2_{\eps}(x) - \sigma^2_{\eps}(x) | & \leq 2^5 \pi \lambda(x) \sup_{\tilde v\in[0,2\sqrt{2\eps} ]} |a''(\tilde v)| \int_0^{r'_\eps} v^{3 - \alpha(x)} dv,\\
& \leq 2^{4} \pi \sup_{x} \lambda(x) \sup_{v\in[0,2\sqrt{2\eps}]} |a''(v)| \, \eps^{2-\alpha_M/2}.
\end{split} \]
Collecting the various estimates on the $D_j$ and using that $\eps\leq \eps_0<1$ concludes the proof of Lemma \ref{lemma:approx_sig} and therefore of Proposition \ref{prop:approx_RTE}.
$\square$
\end{proof}

\subsection{Proof of Proposition \ref{rep_prob2}}\label{proof_rep_prob2}

We first show that the infinitesimal generator of the Markov process $D_\eps$ is $\calA_\eps$.

\paragraph{Infinitesimal generator for $D_\eps$.}  Let $f$ be a smooth bounded function on $\mathbb{R}^3\times \mathbb{S}^2$. The goal of this section is to prove that
\be \label{GA}
\lim_{h\to 0^+}\frac{1}{h}\big(\mathbb{E}_{z}[f(D_\eps(h))] - f(z)\big)= \mathcal{A}_\eps f(z).
\ee
To this end, we introduce the first jump time $T_1$ to obtain
\be \label{decT1}
\mathbb{E}_{z}[f(D_\eps(t))]  = \mathbb{E}_{z}[f(D_\eps(t))\mathbf{1}_{(T_1 > t)}]  + \mathbb{E}_{z}[f(D_\eps(t))\mathbf{1}_{(T_1 \leq t)}].
\ee
Using conditional expectations, we  find for the first term
\[\begin{split}
\mathbb{E}_{z}[f(D_\eps(t))\mathbf{1}_{(T_1 > t)}] & = \mathbb{E}_{z}[f(\psi^{z}_0(t))\mathbf{1}_{(T_1 > t)}]=\mathbb{E}_{z}\big[\mathbb{E}_{z}[f(\psi^{z}_0(t))\mathbf{1}_{(T_1 > t)}|\, \psi^{z}_{0}(s),\,s\in[0,t]]\big]\\
& = \mathbb{E}_{z}\Big[f(\psi^{z}_0(t)) \, \mathbb{P}_{z}\big(T_1 > t \,|\, \psi^{z}_{0}(s),\,s\in[0,t] \big)\Big]= \mathbb{E}_{z}\Big[f(\psi^{z}_{0}(t))e^{-\int_0^t \Lambda_\eps(\psi^{z}_0(s)) ds}\Big].
\end{split}\] 
With the following notations for the flow $\psi^{z}_n$ ,
\[\
\psi^{z}_n  = (X^{x}_{n}, K^{\hat k}_{n}) = \big((X^{x}_{j,n})_{j=1,2,3}, (K^{\hat k}_{j,n})_{j=1,2,3}\big)\in \R^3 \times\mathbb{S}^2,
\]
together with $
\check f(x,k) = f(x,k/|k|)$ for $(x,k)\in \mathbb{R}^3\times \mathbb{R}^3$, the It\^o formula yields
\[\begin{split}
    d f(\psi^{z}_n(t))  = d \check f(\psi^{z}_n(t)) & = \nabla_x \check f (\psi^{z}_n(t)) \cdot dX^{x}_{n}(t)  + \nabla_k \check f (\psi^{z}_n(t)) \cdot dK^{\hat k}_{n}(t)\\
    &\quad + \frac{1}{2} \sum_{j,l=1,2,3} \partial^2_{k_j k_l} \check  f(\psi^{z}_n(t)) d< K^{\hat k}_{j,n}(t), K^{\hat k}_{l,n}(t) > \\
& =  K^{\hat k}_{n}(t)\cdot \nabla_x \check f (\psi^{z}_n(t)) \,dt + \sigma_{\eps}^2(X^{x}_{n}(t)) \Delta_{\mathbb{S}^2} f (\psi^{z}_n(t))\, dt\\
& \quad +  \sqrt{2}\,\sigma_{\eps}(X^{x}_n (t)) \nabla_k \check f (\psi^{z}_{n}(t))\cdot ( K^{\hat k}_n(t) \times dW_n(t)).
\end{split}\]
Above, we have used the fact that
\be
\label{expDS}
\Delta_{\mathbb{S}^2} f(x, \hat k) = \Delta_k \check f(x,\hat k) - \sum_{j,l=1,2,3} \hat k_j \hat k_l \partial^2_{k_jk_l} \check f (x,\hat k) - 2 \sum_{j=1,2,3}  \hat k_j\partial_{k_j} \check f (x,\hat k).
\ee
Therefore, we have for $n=0$,
\[
d\Big((f(\psi^{z}_0(t))-f(z))e^{-\int_0^t \Lambda_\eps(\psi^{z}_0(s)) ds} \Big)  = \Big( d f(\psi^{z}_0(t)) -\Lambda_\eps(\psi^{z}_0(t)) (f(\psi^{z}_0(t))-f(z)) \Big)e^{-\int_0^t \Lambda_\eps(\psi^{z}_0(s)) ds}
\]
so that
\[
\lim_{h\to 0^+}\frac{1}{h}\mathbb{E}_{z}[(f(D_\eps(h)) - f(z))\mathbf{1}_{(T_1 > h)}] = \hat k \cdot \nabla_x f(z) + \sigma_{\eps}^2(x) \Delta_{\mathbb{S}^2} f(z), \qquad \forall z=(x,\hat k)\in \mathbb{R}^3 \times \mathbb{S}^2.
\]

Regarding the second term in \fref{decT1}, we find, using the Markov property in the third line,
\[\begin{split}
\mathbb{E}_{z}\big[f(D_\eps(h))&\mathbf{1}_{(T_1 \leq h)}\big]  = \mathbb{E}_{z}\big[\,\mathbb{E}_{z}[f(D_\eps(h))\,|\, T_1]\, \mathbf{1}_{(T_1 \leq h)}\big] \\
& = \mathbb{E}_{z}\Big[\int_0^h \mathbb{E}_{z}[f(D_\eps(h))\, |\, T_1 =v]\, \Lambda_\eps(\psi^{z}_0(v)) \,e^{-\int_0^v \Lambda_\eps(\psi^{z}_0(s)) ds}  dv\Big] \\
& = \mathbb{E}_{z}\Big[\int_0^h \int_{\mathbb{R}^3\times\mathbb{S}^2}\mathbb{E}_{z'}[f(D_\eps(h-v))] \, \Pi_\eps(\psi^{z}_0(v), dz') \, \Lambda_\eps(\psi^{z}_0(v)) \, e^{-\int_0^v \Lambda_\eps(\psi^{z}_0(s)) ds}  dv \Big]\\
& = \mathbb{E}_{z}\Big[\int_0^h \int_{\mathbb{S}^2}\mathbb{E}_{(X_0^x(h-w),\hat p)}[f(D_\eps(w))] \, \pi_\eps(\psi^{z}_0(h-w),\hat p) \, \sigma(d\hat p)\, \Lambda_\eps(\psi^{z}_0(h-w)) \, e^{-\int_0^{h-w} \Lambda_\eps(\psi^{z}_0(s)) ds}  dw \Big],
\end{split}\]
where the probability $\Pi_\eps$ and the density $\pi_\eps$ are defined respectively in \eqref{def_Pi} and \eqref{def_pi}. As a consequence,
\[
\lim_{h\to 0^+}\frac{1}{h}\mathbb{E}_{z}[f(D_\eps(h)) \mathbf{1}_{(T_1 \leq h)}] = \Lambda_\eps(z) \int_{\mathbb{S}^2} f(x,\hat p) \pi_\eps(z,\hat p) \sigma(d\hat p).
\]
Moreover, we have
\[
\mathbb{P}_{z}(T_1\leq h) = \mathbb{E}_{z}\big[ \mathbb{P}_{z}\big(T_1\leq h\,|\, \psi^{z}_0(s), \,s\in[0,h]\big)\big] = \mathbb{E}_{z}\Big[ 1 - e^{-\int_0^h \Lambda_\eps(\psi^{z}_0(s))ds}\Big],
\]
and it is then direct to see that
\[
\lim_{h\to 0}\frac{1}{h}\mathbb{P}_{z}(T_1\leq h) =  \Lambda_\eps(z).
\]
This finally yields
\[
\lim_{h\to 0^+}\frac{1}{h}\mathbb{E}_{z}[(f(D_\eps(h)) - f(z))\mathbf{1}_{(T_1 \leq h)}] = \Lambda_\eps(z) \int_{\mathbb{S}^2} ( f(x,\hat p) -  f(x,\hat k)) \pi_\eps(z,\hat p) \sigma(d\hat p),
\]
which gives \fref{GA} collecting all results.

\paragraph{Proof of \fref{eq:prob_rep2}.} Since $D_\eps$ is a solution to the martingale problem associated to $\calA_\eps$ (see \cite[Proposition 1.7 pp. 162]{Ethier}), we have, for any smooth bounded function $f$ on $\mathbb{R}^3\times\mathbb{S}^2$,
\[
\mu_t(f) = \mu_0(f) + \int_0^t \mu_s(\calA_\eps f) \,ds, \qquad \textrm{where} \qquad
\mu_t(f) := \mathbb{E}_{\mu_0}[f(D_\eps(t))].
\]
Let
\[
\nu_t(f) := \frac{1}{\bar u_0} \int_{\mathbb{R}^3\times \mathbb{S}^2} u_\eps(t,x,\hat k) f(x,\hat k) \,dx \, \sigma(d\hat k).
\]
Since $u_\eps$ solves \fref{RTE_approx}, we have 
\[u_\eps(t) = u_0 + \int_0^t \calA^\ast_\eps u_\eps(s)ds,\]
where $\calA^\ast_\eps$ stands for the adjoint operator of $\calA_\eps$ in $L^2(\mathbb{R}^3 \times \mathbb{S}^2)$. Then,
\[\begin{split}
    \nu_t(f) &=  \mu_0(f) + \frac{1}{\bar u_0} \int_{\mathbb{R}^3\times \mathbb{S}^2} \int_0^t \calA^\ast_\eps u_\eps(s, x,\hat k) \, ds \, f(x,\hat k)\, dx \, \sigma(d\hat k)\\
    &= \mu_0(f) + \int_0^t \frac{1}{\bar u_0} \int_{\mathbb{R}^3\times \mathbb{S}^2} u_\eps(s, x,\hat k)  \calA_\eps  f(x,\hat k) \, dx \, \sigma(d\hat k) \, ds= \mu_0(f) + \int_0^t \nu_s(\calA_\eps f) ds.
\end{split}\]
Therefore, according to \cite[Proposition 9.18 pp. 251]{Ethier}, we have $ \mu_t = \nu_t $ for any $t\geq 0$, which concludes the proof.

\subsection{Proof of Theorem \ref{mainth}}\label{proof_cv}

The proof of this result is provided in three steps. The first step consists in rewriting the probabilistic representation \eqref{eq:prob_rep2} for \eqref{RTE_approx} is terms of a SDE with jumps. The second step concerns the error analysis of the solution to this later SDE with its discretized version. Finally, the last step gathers all the error estimated and concludes the proof.

\subsubsection{Step 1}

We first introduce an equivalent formulation (in the statistical sense) for the process $D_\eps$ in terms of a stochastic differential equation (SDE) with jumps. This representation is useful when comparing  with the discrete scheme. Let $\tilde D_\eps = (\tilde X_\eps, \tilde  K_\eps)$ be the solution to the following SDE with jumps:
\begin{equation}\label{SDE_jump}\begin{split}
d \tilde X_{\eps}(t) & = \tilde K_{\eps}(t) \, dt\\
d \tilde K_{\eps}(t) & = \sqrt{2}\, \sigma_{\eps}(\tilde X_{\eps}(t^-))\,\tilde K_{\eps}(t^-)\times dW - 2\, \sigma^2_{\eps}(\tilde X_{\eps}(t^-))\tilde K_{\eps}(t^-)\,dt \\
&\quad + \int_{(0,\pi)\times (0,2\pi)\times(0,1)} \tilde R(\theta, \varphi, \tilde D_\eps(t^-), v) P(dt,d\theta,d\varphi,dv)
\end{split}\end{equation}
where the function $\tilde R$ is defined by, for $z=(x,\hat k)$,
  \[
\tilde R(\theta, \varphi,z, v) = \left\{\begin{array}{ccl} R(\theta, \varphi,\hat k) - \hat k &\text{if} & v \leq \Lambda_\eps(z) \pi_\eps(z,R(\theta, \varphi,\hat k))/\bar \Lambda_\eps \\ 
0 & \text{otherwise.}&
\end{array}\right.
\]
Above, $R$ is defined in \eqref{eq:def_R}, $\Lambda_\eps$ in \eqref{eq:Lambda}, $\bar \Lambda_\eps$ in \eqref{eq:def_bar_Lambda}, $\pi_\eps $ in \fref{def_pi}, and $P$ is a random Poisson measure with intensity measure
\begin{equation}\label{eq:def_int_poisson}
\mu(dt,d\theta,d\varphi,du) = \bar \Lambda_\eps \mathbf{1}_{(0,\infty)\times(0,\pi)\times (0,2\pi)\times (0,1)}(t,\theta,\varphi,v)\sin(\theta)dt\,d\varphi\,d\theta /(4\pi).
\end{equation}
See e.g. \cite[Chapter 2]{Applebaum} for more details on Poisson random measures. The notation $t^-$ is standard and refers to the left limit when approaching $t$ before a jump. With this construction, the infinitesimal generator for the Markov process $\tilde D_\eps$ is 
\[
\tilde \calA_\eps f(z) = \hat k \cdot \nabla_x f(z) + \sigma^2_{\eps}(x) \Delta_{\mathbb{S}^2} f(z) + \bar \Lambda_\eps \int_{(0,\pi)\times(0,2\pi)\times (0,1)} (f(x, \hat k + \tilde R(\theta, \varphi,z, v)) - f(z) ) \sigma(d\hat p)dv,
\] 
and we have the following result.
\begin{lemma} \label{lemdis}
The processes $\tilde D_{\eps}$ and $D_{\eps}$ have the same distribution.
\end{lemma}
\begin{proof}
$\tilde D_{\eps}$ and $D_{\eps}$ are both Markov processes and are therefore characterized by their generators. We just have then to prove that $\calA_\eps g= \tilde \calA_\eps g$ for any bounded smooth function $g$. This is a direct consequence of the definition to $\tilde R$. Indeed, denoting by $\tilde \calI_\eps$ and $\calI_\eps$ the integral operators in respectively $\tilde\calA_\eps$ and $\calA_\eps$, we have with $z=(x,\hat k)$,
\[\begin{split}
\tilde \calI_\eps g(z) &= \bar \Lambda_\eps \int_{(0,\pi)\times(0,2\pi)\times (0,1)}  (g(x, \hat k + \tilde R(\theta, \varphi,z, v)) - g(z) ) \, \sin(\theta)\, d\varphi \, d\theta \, dv/(4\pi)\\
 & = \bar \Lambda_\eps \int_{(0,\pi)\times(0,2\pi)} \int_{0}^{\Lambda_\eps(z)\pi_\eps(z,R(\theta, \varphi,\hat k))/\bar \Lambda_\eps} dv \, (g(x, R(\theta, \varphi,\hat k)) - g(z) )\,\sin(\theta)\,d\varphi\, d\theta /(4\pi) \\
& = \bar \Lambda_\eps \int_{\mathbb{S}^2} \frac{\Lambda_\eps(z)\pi_\eps(z,\hat{p})}{\bar \Lambda_\eps}  (g(x,\hat p) - g(z) )\,\sigma(d\hat p) =  \calI_\eps g(z),
\end{split}\]
which concludes the proof.
$\square$ 
\end{proof}

\subsubsection{Step 2} The goal is now to prove that the discretized process $D_{h,\eps}$ approximates $\tilde D_\eps$ in a statistical sense. We use for this the notations of Section \ref{secdifpart} for $X_{n,m}$, $K_{n,m}$ and $\hat K_{n,m}$. For simplicity, we suppose that the Gaussian vectors $(W_{n,m})$ in \fref{def:scheme} are obtained from a single 3D standard Brownian motion $W$ as follows: for $n \geq 0$ and $m=0,\dots,m_n$, we set $W_{n,m}=(W(T_{n,m}+h_{n,m})-W(T_{n,m}))/\sqrt{h_{n,m}}$. In the sequel, we will use the following process, defined by, for $t \in [T_{n,m}, T_{n,m}+h_{n,m}]$, $m=0,\dots,m_n$,
\begin{equation}\left\{\begin{split} \label{SDE_jump2}
\mathcal{X}_{n,m}(t) & =  X_{n,m} + \int_{T_{n,m}}^t \hat K_{n,m} ds \\
\mathcal{K}_{n,m}(t) & = \hat K_{n,m} -2 \int_{T_{n,m}}^t \sigma_{\eps}(X_{n,m}) \hat K_{n,m} ds + \sqrt{2} \int_{T_{n,m}}^t \sigma_{\eps}(X_{n,m}) \, \hat K_{n, m} \times dW(s).
\end{split}\right.\end{equation}
For $t \geq 0$, we then combine the $(\mathcal{X}_{n,m},\mathcal{K}_{n,m})$ into
\[
\mathfrak{D}_{h,\eps} (t)= (  \mathfrak{X}_{h,\eps}(t),  \mathfrak{K}_{h,\eps}(t)) = \sum_{n=0}^\infty\sum_{m=0}^{m_n}\mathbf{1}_{[T_{m,n},T_{m,n}+h_{n,m})}(t) \Psi_{n,m}(t),
\] 
where $\Psi_{n,m}(t) = (\mathcal{X}_{n,m}(t),\mathcal{K}_{n,m}(t))$ is the solution to \fref{SDE_jump2} with initial condition $\Psi_{n,m}(T_{n,m}) = (X_{n,m},\hat K_{n,m})$. 
Note that $\mathfrak{D}_{h,\eps}$ is simply an interpolation of $D_{h,\eps}$ in the intervals $[T_{n,m},T_{n,m}+h_{n,m}]$ that will allow us to use the It\^o formula.

For any smooth function $f$, we now introduce $b$ the (smooth) solution to the following backward RTE,
\begin{equation}\label{eq:back_RTE}
\partial_t b + \hat k \cdot \nabla_x b + \frac{\lambda(x)}{2^{1+\alpha(x)/2}}\int_{\mathbb{S}^2} \rho(x,\hat k \cdot\hat p)(b(\hat p)-b(\hat k)) \sigma(d\hat p) = 0,
\end{equation}
with terminal condition $b(T,x,\hat k) =  f(x,\hat k)$ and use the notation $\check b(t,x,k) = b(t,x,k/|k|)$, $(t,x,k)\in [0,T]\times \mathbb{R}^3 \times \mathbb{R}^3$. We have the following result.

\begin{prop} \label{convE}
For any $T>0$, any smooth bounded function $f$ on $\mathbb{R}^3 \times \mathbb{S}^2$, and any $(x,\hat k )\in\mathbb{R}^3 \times \mathbb{S}^2$, we have
\[
\big| \E_{(x,\hat k)}[f(D_{h,\eps} (T))] -  \E_{(x,\hat k)}[f(\tilde D_\eps (T))] \big| \leq \eps^{2-\alpha_M/2} \, 2T \, E_\infty(b)+h F_1(b),
\]
where $E_\infty(b)$ is defined as in \fref{eq:def_C_eps} with $L^2$ norms replaced by $L^\infty$ norms in all variables, and where $F_1(b)$ is an explicit function independent of $\eps$ and $h$ that depends on derivatives of $b$ up to order 4.
\end{prop}
The notation $\E_{(x,\hat k)}$ above indicates that the process under the expectation starts at the point $(x,\hat k)$.
\begin{proof}
The proof consists in analyzing the discretization error of the diffusion process in a weak sense following the ideas of \cite{talay}.

We first write
\[
E_{h,\eps} := \E_z[f(D_{h,\eps} (T))] -  \E_z[f(\tilde D_\eps (T))]  = \E_z[b(T,D_{h,\eps} (T))] -  \E_z[b(T,\tilde D_\eps (T))],\]
with $z=(x,\hat k)$. Using the It\^o formula, see e.g.  \cite[Th 4.4.7 p. 226]{Applebaum}, together with the expression of $\Delta_{\mathbb{S}^2}$ given in \fref{expDS}, we have
\[\begin{split}
\E_z[ b(T,\tilde{D}_{\eps} (T))] -  b(0,z) = \E_z\Big[\int_0^T  \partial_t b(t,\tilde{D}_{\eps} (t)) & + \tilde K_{\eps}(t)\cdot \nabla_x b(t,\tilde{D}_{\eps} (t)) \\
& + \sigma^2_{\eps}(\tilde X_\eps(t)) \Delta_{\mathbb{S}^2} b(t,\tilde {D}_{\eps} (t)) + \calI^\eps_{>}(t) dt
\Big],
\end{split}\] 
where
\[
\calI^\eps_>(t) = \frac{\lambda(\tilde X_{\eps}(t))}{2^{1+\alpha(\tilde X_{\eps}(t))/2}} \int_{S^\eps_>} \rho(\tilde X_{\eps}(t), \tilde K_{\eps}(t) \cdot \hat p)( b(t,\tilde X_{\eps}(t),\hat p) - b(t,\tilde X_{\eps}(t),\tilde K_{\eps}(t))\sigma(d\hat p).
\]
Using the fact that $b$ satisfies \eqref{eq:back_RTE}, we find
\[
\E_z[ b(T,\tilde{D}_{\eps} (T))]  = b(0,z)+ \E_z\Big[\int_0^T  \sigma^2_{\eps}(\tilde X_\eps(t))\Delta_{\mathbb{S}^2} b(t,\tilde{D}_{\eps} (t)) - \calI^\eps_<(t)    dt \Big],
\]
where $\calI^\eps_<$ is as $\calI^\eps_>$ with $S^\eps_>$ replaced by $S^\eps_<$.
Following the lines of the proof of Lemma \ref{lemma:approx_sig}, we obtain
\[
|\E[ b(T,\tilde{D}_{\eps} (T))] - b(0,z) | \leq \eps^{2-\alpha_M/2} \, T \, E_\infty(b),
\]
where $E_\infty(b)$ is defined as in \fref{eq:def_C_eps} with $L^2$ norms replaced by $L^\infty$ norms. We move now to the term $\E_z[b(T,D_{h,\eps} (T))]$ which requires more work. Decomposing the interval $[0,T]$ according to the grid $(T_{n,m})$, we have
\begin{align*}
b(T, {D}_{h,\eps}(T)) - b(0,z) & = \check b(T, {D}_{h,\eps}(T)) - \check b(0,z) \\
& = \sum_{n\geq 0} \sum_{m\geq 0} \check b(T_{n,m+1}\wedge T,{D}_{h,\eps}(T_{n,m+1}\wedge T)) - \check b(T_{n,m}\wedge T,{D}_{h,\eps}(T_{n,m}\wedge T)) \\
&  \hspace{1cm}+ \check b(T_{n,m}\wedge T, {D}_{h,\eps}(T_{n,m}\wedge T)) - \check b(T_{n,m}\wedge T, {D}_{h,\eps}(T_{n,m}\wedge T^-)) \\
&=: B_1 + B_2,
\end{align*}
with obvious notations and where $B_1$ is meant to capture the dynamic between jumps while $B_2$ captures that at the jumps. The double sum and the $T_{n,m} \wedge T$ are only here to simplify the proof. Note that in order to define the sum for all $m \geq 0$, we set $T_{n,m}=T$ for $m>m_n$, and note also that there is only a finite number of terms in the sums. We are then led to estimate the differences in $B_1$ and $B_2$ for which we will use the process $\mathfrak{D}_{h,\eps}$. We introduced $\check b$ since $\mathfrak{K}_{h,\eps}$ is not necessarily on the sphere between the grid points. Since $T$ is on the grid, we have by definition $D_{h,\eps} (T_{n,m} \wedge T)=\mathfrak{D}_{h,\eps}(T_{n,m} \wedge T)$.
Consider now the notation \[
\hat{\mathfrak{D}}_{\eps,h} := (\mathfrak{X}_{\eps,h}, \hat{\mathfrak{K}}_{\eps,h})
\qquad\text{with}\qquad \hat{\mathfrak{K}}_{\eps,h} := \frac{\mathfrak{K}_{\eps,h}}{|\mathfrak{K}_{\eps,h}|}.
\]
By construction, the process $\hat{\mathfrak{D}}_{h,\eps}$ is continuous at the times $T_{n,m}$ that do not correspond to jump times, so that 
\[
\check b(T_{n,m}, \mathfrak{D}_{h,\eps}(T_{n,m}))=b(T_{n,m}, \hat{\mathfrak{D}}_{h,\eps}(T_{n,m}))=b(T_{n,m}, \hat{\mathfrak{D}}_{h,\eps}(T_{n,m}^-))=\check b(T_{n,m}, \mathfrak{D}_{h,\eps}(T_{n,m}^-))
\]
for those $T_{n,m}$. As a consequence, $B_2$ indeed only accounts for jumps. We will then estimate $B_1$ using the It\^o formula for $\mathfrak{D}_{h,\eps}$ between $T_{n,m}$ and $T_{n,m+1}$, and the properties of Poisson random measures for $B_2$. For the latter, we notice that we have by construction $D_{h,\eps}(\bar T^-_{n})=\hat {\mathfrak{D}}_{h,\eps}(\bar T^-_{n})$. 
Using then the random Poisson measure $P$ with intensity measure $\mu$ introduced in \fref{SDE_jump} and \fref{eq:def_int_poisson}, we can write
\[
B_2 = \int_0^T \big( (\check b(t, \mathfrak{D}_{h,\eps}(t^-)) + \bar R(\theta,\varphi, \hat{\mathfrak{D}}_{h,\eps}(t^-),v) ) - \check b(t, \mathfrak{D}_{h,\eps}(t^-)) \big)  P(dt,d\theta,d\varphi,dv),
\]
so that, together with the fact that $P-\mu$ is a measure-valued martingale, see e.g. \cite[Chapter 2]{Applebaum},
\[
\E_z[B_2]= \sum_n \sum_m \E_z\Big[\int_{T_{n,m}\wedge T}^{T_{n,m+1}\wedge T} \tilde \calI^\eps_>(t)dt \Big]
\]
with
\[
\tilde \calI^\eps_>(t) = \frac{\lambda(\mathfrak{X}_{h,\eps}(t))}{2^{1+\alpha(\mathfrak{X}_{h,\eps}(t))/2}} \int_{S^\eps_>} \rho(\mathfrak{X}_{h,\eps}(t), \hat{\mathfrak{K}}_{h,\eps}(t)\cdot \hat p)( \check b(t,\mathfrak{X}_{h,\eps}(t),\hat p) - \check b(t,\mathfrak{X}_{h,\eps}(t),\mathfrak{K}_{h,\eps}(t))\sigma(d\hat p).
\]
%
%
%
For $B_1$, we have from the It\^o formula 
\[
\E_z[B_1] = \sum_n \sum_m \E_z\Big[\int_{T_{n,m}\wedge T}^{T_{n,m+1}\wedge T}  \partial_t \check b(t,\mathfrak{D}_{h,\eps} (t))  + \mathfrak{K}_{h,\eps}(T_{n,m})\cdot \nabla_x \check b(t,\mathfrak{D}_{h,\eps} (t)) + \sigma^2_{\eps}(\mathfrak{X}_{h,\eps}(T_{n,m})) \calB^\eps_{n,m}(t) \, dt \Big]
\]
where
\[
\calB^\eps_{n,m}(t) = \Delta_k  \check b(t,\mathfrak{D}_{h,\eps} (t)) - \mathfrak{K}_{h,\eps} (T_{n,m})^T D^2_k \check b(t,\mathfrak{D}_{h,\eps} (t))\mathfrak{K}_{h,\eps} (T_{n,m}) -2 \mathfrak{K}_{h,\eps} (T_{n,m}) \cdot \nabla_k \check b(t,\mathfrak{D}_{h,\eps} (t)).  
\] 
As a result,
\begin{align*}
\E_z[B_1 + B_2] = \sum_n \sum_m \E_z\Big[\int_{T_{n,m}\wedge T}^{T_{n,m+1}\wedge T}  & \partial_t \check b(t,\mathfrak{D}_{h,\eps} (t))  + \mathfrak{K}_{h,\eps}(T_{n,m})\cdot \nabla_x \check b(t,\mathfrak{D}_{h,\eps} (t)) \\
&+ \sigma^2_{\eps}(\mathfrak{X}_{h,\eps}(T_{n,m})) \calB^\eps_{n,m}(t) + \tilde \calI^\eps_>(\mathfrak{D}_{h,\eps}(t))\, dt \Big].
\end{align*}
Using again the fact that $b$ satisfies \eqref{eq:back_RTE}, we have
\[
\partial_t b(t,\mathfrak{D}_{h,\eps} (T_{n,m}))  + \mathfrak{K}_{h,\eps}(T_{n,m})\cdot \nabla_x  b(t,\mathfrak{D}_{h,\eps} (T_{n,m})) + \calQ b(t,\mathfrak{D}_{h,\eps} (T_{n,m})) = 0,
\]
which also holds true for $\check b$ since the variable $\mathfrak{K}_{h,\eps}(T_{n,m})$ has norm $1$ at the grid points $T_{n,m}$. As a consequence,
\[
\E_z[b(T,D_{h,\eps} (T))] - b(0,z) = \sum_n \sum_m \sum_{j=1}^7 E^{(j)}_{n,m}
\]
where
\begin{align*}
E^{(1)}_{n,m} & := \E_z\Big[\int_{T_{n,m}\wedge T}^{T_{n,m+1}\wedge T}   \partial_t \check b(t,\mathfrak{D}_{h,\eps} (t)) - \partial_t  \check b(t,\mathfrak{D}_{h,\eps} (T_{n,m})) \, dt \Big]\\
E^{(2)}_{n,m} & :=\E_z\Big[\int_{T_{n,m}\wedge T}^{T_{n,m+1}\wedge T} \mathfrak{K}_{h,\eps}(T_{n,m})  \cdot (\nabla_x  \check b(t,\mathfrak{D}_{h,\eps} (t)) - \nabla_x \check b(t,\mathfrak{D}_{h,\eps} (T_{n,m})) ) \, dt \Big]\\
E^{(3)}_{n,m} & := \E_z\Big[\int_{T_{n,m}\wedge T}^{T_{n,m+1}\wedge T} \sigma^2_{\eps}(\mathfrak{X}_{h,\eps}(T_{n,m}))\, (\Delta_k \check b(t,\mathfrak{D}_{h,\eps} (t)) -\Delta_k \check b(t,\mathfrak{D}_{h,\eps} (T_{n,m})) ) \, dt\Big]\\
E^{(4)}_{n,m} & := \E_z\Big[\int_{T_{n,m}\wedge T}^{T_{n,m+1}\wedge T} - \sigma^2_{\eps}(\mathfrak{X}_{h,\eps}(T_{n,m}))\mathfrak{K}_{h,\eps} (T_{n,m})^T ( D^2_k \check b(t,\mathfrak{D}_{h,\eps} (t)) - D^2_k \check b(t,\mathfrak{D}_{h,\eps} (T_{n,m})) ) \mathfrak{K}_{h,\eps} (T_{n,m}) \, dt \Big]\\
E^{(5)}_{n,m} & :=\E_z\Big[\int_{T_{n,m}\wedge T}^{T_{n,m+1}\wedge T} -2 \, \sigma^2_{\eps}(\mathfrak{X}_{h,\eps}(T_{n,m}))\, \mathfrak{K}_{h,\eps} (T_{n,m}) \cdot (\nabla_k \check b(t,\mathfrak{D}_{h,\eps} (t)) - \nabla_k\check b(t,\mathfrak{D}_{h,\eps} (T_{n,m})) ) \, 
dt\Big]\\ 
E^{(6)}_{n,m} & := \E_z\Big[\int_{T_{n,m}\wedge T}^{T_{n,m+1}\wedge T}  \sigma^2_{\eps}(\mathfrak{X}_\eps(T_{n,m}))\Delta_{\mathbb{S}^2} b(t,\mathfrak{D}_{h, \eps} (T_{n,m})) - \tilde \calI^\eps_<(\mathfrak{D}_{h,\eps}(T_{n,m}))  \,  dt \Big]\\
E^{(7)}_{n,m} & := \E_z\Big[\int_{T_{n,m}\wedge T}^{T_{n,m+1}\wedge T}  \tilde \calI^\eps_>(\mathfrak{D}_{h,\eps}(t)) - \tilde \calI^\eps_>(\mathfrak{D}_{h,\eps}(T_{n,m}))  \,  dt \Big]
\end{align*}
with
\be \label{def:I_tilde}
\tilde \calI^\eps_<(\mathfrak{D}_{h,\eps}(t)) = \frac{\lambda(\mathfrak{X}_{h,\eps}(t))}{2^{1+\alpha(\mathfrak{X}_{h,\eps}(t))/2}} \int_{S^\eps_<} \rho(\mathfrak{X}_{h,\eps}(t), \hat{\mathfrak{K}}_{h,\eps}(t)\cdot \hat p)( \check b(t,\mathfrak{X}_{h,\eps}(t),\hat p) - \check b(t,\mathfrak{D}_{h,\eps}(t))\sigma(d\hat p).
\ee

In the estimates below involving $\check b$, derivatives involving $k\mapsto \hat k = k/|k|$ produce terms of the form $1/|k|^p$ for some $p>0$. These terms are due to the fact that $\mathfrak{K}_{n,m}$ does not stay on the sphere at all times. However, these terms can be bounded uniformly thanks to the following lemma. 

\begin{lemma}
We have for any $n \geq 0$ and $m \leq m_n$,
\[
\inf_{s\in [T_{n,m},T_{n,m}+h_{n,m})} |\mathcal{K}_{n,m}(s)| \geq (1-2h\sigma^2_{\eps,\infty}) \geq \frac{1}{2},
\]
for $h$ and $\eps$ small enough, and
\[
\sup_{s\in [T_{n,m},T_{n,m}+h_{n,m})} \E\big [|\mathcal{K}_{n,m}(s)|^2 \, | \, T_{n,m}, T_{n,m+1} \big] \leq 1 + 4 h \sigma^2_{\eps,\infty}\leq 2.
\]
\end{lemma}
 
\begin{proof}
For $s \in [T_{n,m},T_{n,m}+h_{n,m}]$, $\mathcal{K}_{n,m}$ can be rewritten as the sum of two orthogonal components
\[
\mathcal{K}_{n,m}(s)  = (1 -2 (t-T_{n,m}) \sigma^2_{\eps}(X_{n,m}) ) \hat K_{n,m} + \sqrt{2} \sigma_{\eps}(X_{n,m}) \, \hat K_{n, m} \times (W(s)-W(T_{n,m}))
\]
so that
\[
|\mathcal{K}_{n,m}(s)|^2 \geq (1 - 2 h \sigma^2_{\infty,\eps})^2.
\]
Now for the upper bound, using that $W_n$ and $P$ are independent, we have
\begin{align*}
\E\big [ |\mathcal{K}_{n,m}(s)|^2 \, | \, T_{n,m}, T_{n,m+1} \big] & \leq (1 + 2 h \sigma^2_{\eps,\infty} ) + 2 \sigma^2_{\eps,\infty} \,\E\big [   |W(s)-W(T_{n,m})|^2\, | \, T_{n,m}, T_{n,m+1} \big] \\
& \leq  1 + 4 h \sigma^2_{\eps,\infty}\\
&\leq 2,
\end{align*}
where we used that $4 h \sigma^2_{\eps,\infty} \leq 1$. This concludes the proof. $\square$ 
\end{proof} 

\paragraph{For $E^{(1)}_{n,m}$.} Using the It\^o formula between $T_{n,m}$ and $T_{n,m+1}$, we find
\begin{align*}
\E\Big[\partial_t \check b(t,\mathfrak{D}_{h,\eps} (t)) & - \partial_t \check b(t,\mathfrak{D}_{h,\eps} (T_{n,m}))\,\Big|\,T_{n,m}, T_{n,m+1}\Big] \\
& = \int_{T_{n,m}}^t  \mathfrak{K}_{h,\eps} (T_{n,m})\cdot\nabla_x  \partial_t  \check b(t,\mathfrak{D}_{h,\eps} (s)) \\
& \;+ \sigma^2_{\eps}(\mathfrak{X}_{h,\eps}(T_{n,m})) \Big(\Delta_k  \partial_t \check b(t,\mathfrak{D}_{h,\eps} (s)) - \mathfrak{K}_{h,\eps} (T_{n,m})^T D^2_k \partial_t \check b(t,\mathfrak{D}_{h,\eps} (s))\mathfrak{K}_{h,\eps} (T_{n,m}) \\
& \;-2 \mathfrak{K}_{h,\eps} (T_{n,m}) \cdot \nabla_k \partial_t \check b(t,\mathfrak{D}_{h,\eps} (s))\Big) ds,
\end{align*}
so that
\[
|E^{(1)}_{n,m}| \leq h \, \Big( \|D^2_{t,x} \check b\|_\infty + 2 \sigma_{\eps,\infty}^2(\|D^3_{t,k,k} \check b\|_\infty +  \| D^2_{t,k} \check b\|_\infty) \Big),
\]
where $
\sigma_{\infty,\eps} = \sup_x \sigma_{\eps}(x).
$
When $\eps \leq \eps_0$, we have $\sigma_{\infty,\eps}\leq \sigma_{\infty,\eps_0}$ since $r_\eps'$ is an increasing function of $\eps$. As a result, we obtain
\[
\sum_n\sum_m |E^{(1)}_{n,m}|  \leq h \, T \, \Big( \|D^2_{t,x} \check b\|_\infty + 2 \sigma_{\infty,\eps_0}^2(\|D^3_{t,k,k} \check b\|_\infty +  \| D^2_{t,k} \check b\|_\infty) \Big)=:h F_2(b).
\]

\paragraph{For $E^{(j)}_{n,m}$ with $j=2,\dots,5$.}

Following the same lines as above, we have
\begin{align*}
\sum_n\sum_m \sum_{j=2}^5 |E^{(j)}_{n,m}|  & \leq 2 h \, T \,  \Big( \|D^2_{x,x} \check b\|_\infty +2 \sigma_{\infty,\eps_0}^2(\|D^3_{x,k,k} \check b\|_\infty + \|D^2_{x,k} \check b\|_\infty) \Big) \\
& + 4 h \, T \, \sigma_{\infty,\eps_0}^2 \Big( \|D^2_{x,k} \check b\|_\infty+  \|D^3_{x,k,k} \check b\|_\infty +2 \sigma_{\infty,\eps_0}^2(2\|D^3_{k,k,k} \check b\|_\infty + \|D^4_{k,k,k,k} \check b\|_\infty+ \|D^2_{k,k} \check b\|_\infty) \Big)\\
&=: h F_3(b).
\end{align*} 

\paragraph{For $E^{(6)}_{n,m}$.}
For this term, we follow the proof of Lemma \ref{lemma:approx_sig} and find 
\[
\sum_n\sum_m|E^{(6)}_{n,m}| \leq T \, \eps^{2-\alpha_M/2} \, E_\infty(b),
\]
where $E_\infty(b)$ is defined as before.

\paragraph{For $E^{(7)}_{n,m}$.} 

Starting from \eqref{def:I_tilde}, we have
\begin{align*}
\tilde \calI^\eps_>(\mathfrak{D}_{h,\eps}(t)) & = \frac{\lambda(\mathfrak{X}_{h,\eps}(t))}{2^{1+\alpha(\mathfrak{X}_{h,\eps}(t))/2}} \int_{S^\eps_>} \rho(\mathfrak{X}_{h,\eps}(t), \hat{\mathfrak{K}}_{h,\eps}(t)\cdot \hat p)( \check b(t,\mathfrak{X}_{h,\eps}(t),\hat p) - \check b(t,\mathfrak{X}_{h,\eps}(t),\hat{ \mathfrak{K}}_{h,\eps}(t))\sigma(d\hat p)\\
& = \frac{\lambda(\mathfrak{X}_{h,\eps}(t))}{2^{1+\alpha(\mathfrak{X}_{h,\eps}(t))/2}} \int_0^{2\pi} d\varphi \int_{-1}^{1-\eps} ds\, \rho(\mathfrak{X}_{h,\eps}(t), s)\\
&\times \Big( \check b\big(t,\mathfrak{X}_{h,\eps}(t), s\hat{\mathfrak{K}}_{h,\eps}(t) + \sqrt{1-s^2}\, \calG(\varphi,\hat{\mathfrak{K}}_{h,\eps}(t)) \big) - \check b\big(t,\mathfrak{X}_{h,\eps}(t),\hat{ \mathfrak{K}}_{h,\eps}(t) \big) \Big) 
\end{align*}
where the last line is obtained by changing to spherical coordinates with $s=\cos(\theta)$, and 
\[
\calG(\varphi, k) := \cos(\varphi)\hat k^\perp_1 + \sin(\varphi)\hat k^\perp_2.
\]
Above, $(\hat k^\perp_1, \hat k^\perp_2)$ forms an orthonormal basis of the plane $\hat k^\perp$. Note that the choice of $(\hat k^\perp_1, \hat k^\perp_2)$ does not play any role since the variable $\varphi$ is integrated. Now, writing
\begin{align*}
\tilde \calI^\eps_>(\mathfrak{D}_{h,\eps}(t))   = \frac{\lambda(\mathfrak{X}_{h,\eps}(t))}{2^{1+\alpha(\mathfrak{X}_{h,\eps}(t))/2}} & \int_0^{2\pi} d\varphi \int_{-1}^{1-\eps} ds\int_0^1  dv\, \rho(\mathfrak{X}_{h,\eps}(t), s)\\
&\times \big((s-1)\hat{\mathfrak{K}}_{h,\eps}(t) + \sqrt{1-s^2}\,\calG(\varphi, \hat{\mathfrak{K}}_{h,\eps}(t)) \big)\\
& \cdot \nabla_k \check b \big(t,\mathfrak{X}_{h,\eps}(t), (1+v(s-1))\hat{\mathfrak{K}}_{h,\eps}(t) + v\sqrt{1-s^2}\calG(\varphi, \hat{\mathfrak{K}}_{h,\eps}(t)) \big),
\end{align*}
and using that $\calG(\varphi + \pi, k)=-\calG(\varphi, k)$, we just have to focus on
\begin{align*}
\tilde \calI^\eps_>(\mathfrak{D}_{h,\eps}(t))   = \frac{\lambda(\mathfrak{X}_{h,\eps}(t))}{2^{1+\alpha(\mathfrak{X}_{h,\eps}(t))/2}} & \int_0^{2\pi} d\varphi \int_{-1}^{1-\eps} ds\int_0^1  dv\, \rho(\mathfrak{X}_{h,\eps}(t), s) (s-1)\hat{\mathfrak{K}}_{h,\eps}(t)\\
& \cdot \nabla_k \check b \big(t,\mathfrak{X}_{h,\eps}(t), (1+v(s-1))\hat{\mathfrak{K}}_{h,\eps}(t) + v\sqrt{1-s^2}\, \calG(\varphi, \hat{\mathfrak{K}}_{h,\eps}(t)) \big).
\end{align*}
Before applying the It\^o formula to this term, we rewrite $\calG$ as
\[
\calG(\varphi,\hat k) = \Big(I_3 + \sin(\varphi)Q(\hat k) + (1-\cos(\varphi))Q^2(\hat k)\Big)\calH_1(\hat k)
\]
where $I_3$ is the $3\times 3$ identity matrix, $Q$ is defined by \eqref{def:Q_mat}, and where 
\[
\calH_1(\hat k) := \frac{1}{\sqrt{\hat k_1^2 + \hat k_2^2}}
\begin{pmatrix} 
\hat k_2\\
-\hat k_1\\
0
\end{pmatrix} = \frac{1}{\sqrt{k_1^2 + k_2^2}}
\begin{pmatrix} 
k_2\\
-k_1\\
0
\end{pmatrix},
\]
which is orthogonal to $\hat k = k/|k|$. 
In fact, $\calG(\varphi,\hat{k})$ corresponds to the rotation of $\calH_1(\hat k)\in\hat k^\perp$ with angle $\varphi$ and axis $\hat k$. This choice simplifies calculations. 
Now, note that
\begin{align*}
 \int_{-1}^{1-\eps}  \rho(x, s) (s-1) ds & \leq \int_{-1}^{1} \rho(x, s) (s-1) \leq \|a\|_\infty \frac{2^{1-\eta_m/2}}{1-\eta_M/2},
\end{align*}
so that $E^{(7)}_{n,m}$ does not depends on $\eps$. Applying the It\^o formula, we obtain 
\begin{align*}
\sum_{n}\sum_{m} |E^{(7)}_{n,m}| \leq & h T \, (1+(1+2^7 \cdot 3^3)\sigma^2_{\infty,\eps_0}) \, \|a\|_\infty \frac{2^{1-\eta_m/2}}{1-\eta_M/2}\Big(1+\frac{1}{1-\eta_M/2} \Big)\\
& \times  \big(\|\nabla_x \lambda\|_\infty + \|\lambda\|_\infty + \|\nabla_x \alpha\|_\infty  \big)\big( \|D^1_k \check b\|_\infty + \|D^2_{k,k} \check b\|_\infty \big) =: h F_4(b)
\end{align*}
Setting finally $F_1:=F_2+F_3+F_4$ and gathering all previous results concludes the proof of Proposition \ref{convE}. $\blacksquare$
\end{proof}

\subsubsection{Step 3 and conclusion}

We remark first that the error bound in Proposition \ref{convE} does not depend on the starting point $(x,\hat k)$. Then, from this \emph{pointwise} result, we find
\[
\big| \E_{\mu_0}[f(D_{h,\eps} (T))] -  \E_{\mu_0}[f(\tilde D_\eps (T))] \big| \leq \int_{\mathbb{R}^3 \times \mathbb{S}^2} \big| \E_{x,\hat k}[f(D_{h,\eps} (T))] -  \E_{x,\hat k}[f(\tilde D_\eps (T))] \big| \mu_0(dx, d\hat k),
\]
where $\mu_0$ is the probability measure given by \eqref{eq:def_mu0}. Let now
\[
\mu(t,f) = \int_{\mathbb{R}^3\times \mathbb{S}^2} u(t, x, \hat k) f(x,\hat k )\,dx\,\sigma(d\hat k)\qquad\text{and}\qquad \mu_{h,\eps}(t,f) = \E_{\mu_0}[f( D_{h,\eps}(t))].
\]
Using Propositions \ref{prop:approx_RTE} and \ref{convE}, and that $\int_{\mathbb{R}^3\times \mathbb{S}^2} u_\eps(T, x, \hat k) f(x,\hat k )\,dx\,\sigma(d\hat k)= \E_{\mu_0}[f(D_\eps (T))]=\E_{\mu_0}[f(\tilde D_\eps (T))]$ according to Lemma \ref{lemdis}, we have
\[
|\mu_{h,\eps}(T,f) - \mu(T,f)| \leq \eps^{2-\alpha_M/2} \,\sqrt{2T} E(u) \|f\|_{L^2(\R^3 \times \bbS^2)} + \eps^{2-\alpha_M/2} \,2T E_\infty(b) + h F_1(b).
\]
In order to end the proof of Theorem \ref{mainth}, it suffices to remark now that
\[
|\mu_{N, h, \eps}(T,f)-\mu(T,f)| \leq |\mu_{N, h, \eps}(T,f)-\mu_{h,\eps}(T,f)|+|\mu_{h,\eps}(T,f)-\mu(T,f)|,
\]
so that
\[\begin{split}
\mathbb{P}\Big(& |\mu_{N, h, \eps}(T,f)-\mu(T,f)|  > \frac{\eta \Sigma_{h,\eps}}{ \sqrt{N}} + \eps^{2-\alpha_M/2} F_0(u,b,f)+ \eps^{2-\alpha_M/2} \,2T E_\infty(b)+ h \, F_{1}(b)\Big) \\
& \leq \mathbb{P}\Big(|\mu_{N, h, \eps}(T,f)-\mu_{h,\eps}(T,f)|> \frac{\eta \Sigma_{h,\eps}}{ \sqrt{N}} + \underbrace{\eps^{2-\alpha_M/2} F_0(u,b,f) + h \,F_1(b)- |\mu_{h,\eps}(T,f)-\mu(T,f)|}_{\geq 0}\Big) \\ 
& \leq \mathbb{P}\Big(|\mu_{N, h, \eps}(T,f)-\mu_{h,\eps}(T,f)|> \frac{\eta \Sigma_{h,\eps}}{ \sqrt{N}}\Big)
\end{split}\]
where
\[
F_0(u,b,f):=\sqrt{2T} E(u) \|f\|_{L^2(\R^3 \times \bbS^2)} + 2T E_\infty(u).
\]
We conclude by applying the central limit theorem \cite{Feller} together with the Portmanteau theorem \cite[Theorem 2.1 pp.16]{Billingsley}. $\blacksquare$

\section{Conclusion}

We have derived an efficient MC method for the resolution of the RTE with non-integrable scattering kernels. It is based on a small jumps/large jumps decomposition that allows us to simulate the small jumps part at a low cost by solving a standard SDE. The large jumps are obtained by using the stochastic collocation technique with a candidate distribution function that captures the singular behavior of the kernel.
We have moreover demonstrated the necessity to include the small jumps component in order to obtain a good accuracy at a manageable computational cost, and investigated practical situations in optical tomography and atmospheric turbulence where the singular RTE is of interest. We in particular highlighted the role of the singularity strength $\alpha$ on the qualitative behavior of the solution.

Future investigations include the estimation of the scattering kernel, with an emphasis on the parameter $\alpha$, from either simulated or experimental data obtained e.g. from light propagation in biological tissues. This problem is of practical interest in biomedical applications and will require the development of appropriate inverse techniques.

\appendix

\section{Stochastic collocation}\label{sec_sto_col}

In this section, we describe the stochastic collocation method, see e.g. \cite{grzelak}, and consider the situation of Section \ref{sec:NKt} as an illustration. The goal is to simulate a real-valued random variable $W$ (for which direct simulation is not possible or too costly) from an auxiliary variable $V$ that can be generated efficiently. In our context, we want to simulate $W$ with probability density function (PDF)
\[
f_W(w) := \frac{a(\sqrt{2w})}{C_W w^{1+\alpha/2}} \mathbf{1}_{(\eps,2)}(w),
\]
where $C_W$ is a normalization constant. As already noticed in Section \ref{sec:MC_algo}, a direct method is available when  $a\equiv 1$. Therefore, we take $V$ with PDF 
\[
f_V(v) := \frac{1}{C_V v^{1+\alpha/2}} \mathbf{1}_{(\eps,2)}(v),
\]
that can be simulated with
\[V = F^{-1}_V(U) = \eps(1- (1 - (\eps /2 )^{\alpha/2}) U )^{-2/\alpha},\]
where $U\sim\calU(0,1)$ and where $F_V$ is the cumulative distribution function (CDF) of $V$. 
The stochastic collocation method is based on the following three observations. First, we have $F_V(V) \sim \calU(0,1)$. 
Second, denoting by $F_W$ the CDF of $W$, we note that $W$ can be (theoretically) simulated with
\[ F^{-1}_W ( U) = F^{-1}_W ( F_V(V)) =: G(V),\]
with $G = F^{-1}_W \circ F_V$ and $U=F_V(V)$.
Last, we only need to approximate $G$ and not $F^{-1}_W$, and with a good candidate $V$, $G$ behaves better than $F^{-1}_W$. In order to approximate $G$, we use Gauss polynomial interpolation and only need to invert $F_W$ at a small number of points. 

In our example, $V$ captures the "singular" behavior of $W$, and is as a consequence a good candidate. The function $G$ is then direct to approximate with just a few quadrature points for a reduced computational cost. Because $G$ needs only to be approximated over $(\eps,2)$, with known values at the extremes, we rather use a Gauss-Lobatto-Jacobi quadrature rule. In Figure \ref{fig:quad}, we illustrate the polynomial approximation of $G$, with 5 and 10 interpolation points for $\alpha=5/3$, $\eps=0.01$, and $a(r) = \exp(-r^2/(2\times 0.8^2)).$ Because of our choice for $V$, one can observe that the overall behavior of the PDF $f_W$ is well captured with just 5 quadrature points, even for strongly singular kernels with $\alpha=5/3$. However, the fast decay of the function $a$, which is the main source of error between $F_W$ and $F_V$, requires more quadrature points for an accurate approximation and 10 points seem sufficient. 

\begin{figure}[h!]
\begin{center}
\includegraphics[scale=0.35]{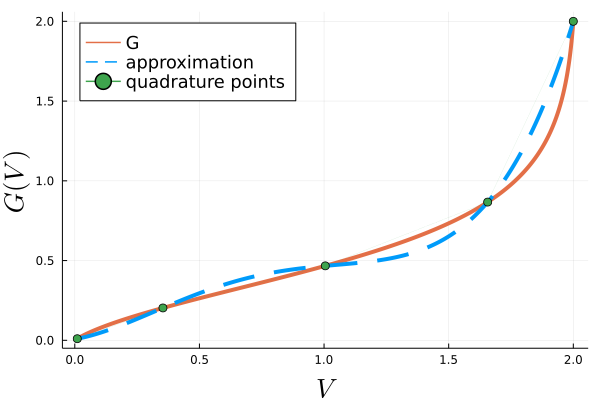}
\includegraphics[scale=0.35]{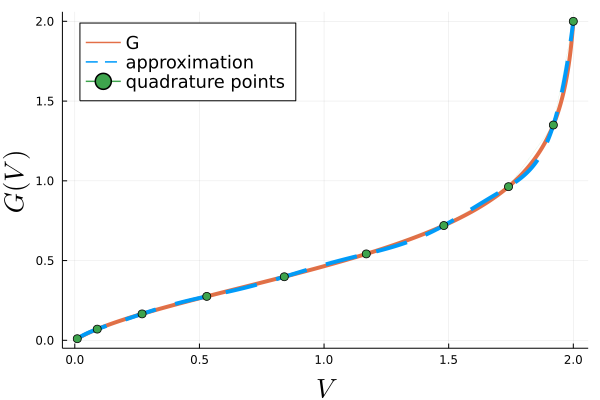} 
\end{center}
\caption{\label{fig:quad} Illustration of the polynomial approximation to $G$ with 5 (left picture) and 10 (right picture) interpolations points. We use the library Jacobi.jl to compute these quadrature points.}
\end{figure}

\small{
}

\end{document}